\documentclass[a4paper,12pt]{article}

\usepackage[utf8]{inputenc}
\usepackage[english]{babel}
\usepackage{amsmath, amssymb, amsthm}
\usepackage{enumerate}
\usepackage[colorlinks=true,linkcolor=blue,citecolor=blue]{hyperref}

\usepackage{tikz}
\usetikzlibrary{fadings}
\tikzfading[name=fade out,
            inner color=transparent!0,
            outer color=transparent!100]
\usetikzlibrary{calc,arrows,decorations.pathreplacing,fadings,3d,positioning}

\title{Multiparametric geometry of numbers \\ and its application to \\ splitting transference theorems.
       \thanks{This research was supported in part by RSF grant 18-41-05001}}
\author{Oleg\,N.\,German}
\date{}

\theoremstyle{definition}
\newtheorem{definition}{Definition}

\newtheorem*{notation*}{Notation}

\theoremstyle{remark}
\newtheorem{remark}{Remark}
\newtheorem*{remark*}{Remark}

\theoremstyle{plain}
\newtheorem{theorem}{Theorem}
\newtheorem{lemma}{Lemma}

\newtheorem{proposition}{Proposition}
\newtheorem{corollary}{Corollary}

\newtheorem*{statement*}{Statement}
\newtheorem*{corollary*}{Corollary}

\DeclareMathOperator{\vol}{vol}

\renewcommand{\phi}{\varphi}

\renewcommand{\vec}[1]{\mathbf{#1}}
\renewcommand{\geq}{\geqslant}
\renewcommand{\leq}{\leqslant}

\newcommand{\e}{\varepsilon}

\newcommand{\R}{\mathbb{R}}
\newcommand{\Z}{\mathbb{Z}}

\newcommand{\La}{\Lambda}

\newcommand{\bpsi}{\underline{\psi}}
\newcommand{\apsi}{\overline{\psi}}
\newcommand{\bPsi}{\underline{\Psi}}
\newcommand{\aPsi}{\overline{\Psi}}

\newcommand{\cB}{\mathcal{B}}

\newcommand{\cH}{\mathcal{H}}

\newcommand{\cP}{\mathcal{P}}
\newcommand{\cQ}{\mathcal{Q}}

\newcommand{\cT}{\mathcal{T}}

\newcommand{\tr}[1]{{#1}^\top}

\begin{document}

\maketitle

\begin{abstract}
  In this paper we consider a multiparametric version of Wolfgang Schmidt and Leonard Summerer's parametric geometry of numbers. We apply this approach in two settings: the first one concerns weighted Diophantine approximation, the second one concerns Diophantine exponents of lattices. In both settings we use multiparametric approach to define intermediate exponents. Then we split the weighted version of Dyson's transference theorem and an analogue of Khintchine's transference theorem for Diophantine exponents of lattices into chains of inequalities between the intermediate exponents we define basing on the intuition provided by the parametric approach.
\end{abstract}

%\noindent
%\textbf{Key words:}

\section{Introduction}

Parametric geometry of numbers was introduced several years ago by W.\,M.\,Schmidt and L.\,Summerer in \cite{schmidt_summerer_2009}, \cite{schmidt_summerer_2013}. It allowed to look at Diophantine problems from a different angle and gave a strong impulse to the development of Diophantine approximation. In this paper we consider a slightly more general setting, which we prefer to call multiparametric.

\subsection{Multiparametric geometry of numbers}

Let $\La$ be a full rank lattice in $\R^d$ of covolume $1$. Let $|\cdot|$ denote the supremum norm. Set
\[\cB=\Big\{ \vec z\in\R^d\, \Big|\,|\vec z|\leq1 \Big\},
  \qquad
  \cT=\Big\{ \pmb\tau=(\tau_1,\ldots,\tau_d)\in\R^d\, \Big|\,\tau_1+\ldots+\tau_d=0 \Big\},\]
and for each $\pmb\tau\in\cT$ set
\[\cB_{\pmb\tau}=
  \textup{diag}(e^{\tau_1},\ldots,e^{\tau_d})
  \cB.\]
Let $\lambda_k(\cB_{\pmb\tau})=\lambda_k(\cB_{\pmb\tau},\La)$, $k=1,\ldots,d$, denote the $k$-th successive minimum, i.e.
the infimum of positive $\lambda$ such that $\lambda\cB_{\pmb\tau}$ contains at least $k$ linearly independent vectors of $\La$.
Finally, for each $k=1,\ldots,d$, let us set
\[L_k(\pmb\tau)=L_k(\La,\pmb\tau)=\log\big(\lambda_k(\cB_{\pmb\tau},\La)\big),\qquad
  S_k(\pmb\tau)=S_k(\La,\pmb\tau)=\sum_{1\leq j\leq k}L_j(\La,\pmb\tau).\]

Many problems in Diophantine approximation can be interpreted as questions concerning the asymptotic behaviour of $L_k(\pmb\tau)$ and $S_k(\pmb\tau)$. Different problems require different subsets of $\cT$ along which $\pmb\tau$ is supposed to tend to infinity. In this paper we show that Diophantine approximation with weights requires considering one-dimensional and two-dimensional subspaces of $\cT$, whereas the whole $\cT$ equipped with appropriate exhaustion leads us to Diophantine exponents of lattices. In both those settings there exist transference theorems. A particular aim of this paper is to apply parametric geometry of numbers to split those transference theorems into chains of inequalities between intermediate exponents. The language of parametric geometry of numbers appears to be very well fit for this purpose. The main tools are provided by Proposition \ref{prop:essence_of_transference_split_up} (see Section \ref{sec:properties_of_L_and_S}), which we believe to be of interest in itself.

In the case considered by W.\,M.\,Schmidt and L.\,Summerer there is a very strong result by D.\,Roy \cite{roy_annals_2015}, which allows considering instead of $L_k(\pmb\tau)$, $S_k(\pmb\tau)$ certain functions obeying rather simple formal laws. It is a challenging problem to obtain an analogue of Roy's theorem for the multiparametric setting. However, for the purposes of the current paper, statements like Proposition \ref{prop:essence_of_transference_split_up} describing local behaviour of $L_k(\pmb\tau)$ and $S_k(\pmb\tau)$ are already enough.

\subsection{Diophantine approximation with weights}\label{sec:weighted_DA}

Given a matrix
\[\Theta=
  \begin{pmatrix}
    \theta_{11} & \cdots & \theta_{1m} \\
    \vdots & \ddots & \vdots \\
    \theta_{n1} & \cdots & \theta_{nm}
  \end{pmatrix}
  \in\R^{n\times m},\ \
  n+m=d,\]
and a real $\gamma$, it is questioned in the most classical `non-weighted' setting whether the system of inequalities
\begin{equation}\label{eq:trivially_weighted_system}
  \begin{cases}
    |\vec x|^m\leq t \\
    |\Theta\vec x-\vec y|^n\leq t^{-\gamma}
  \end{cases}
\end{equation}
admits nonzero solutions in $(\vec x,\vec y)\in\Z^m\oplus\Z^n$ for large values of $t$. Here, as before, $|\cdot|$ denotes the supremum norm.

In multiplicative Diophantine approximation the supremum norm is replaced with the geometric mean.
For instance, the famous Littlewood conjecture, one of the most challenging problems in multiplicative Diophantine approximation, asserts that for every $\theta_1,\theta_2\in\R$ and every $\e>0$ there are arbitrarily large $t$ such that the system of inequalities
\begin{equation*}%\label{eq:littlewood}
  \begin{cases}
    |x|\leq t \\
    |\theta_1x-y_1|\cdot|\theta_2x-y_2|\leq\e t^{-1}
  \end{cases}
\end{equation*}
admits nonzero solutions in $(x,y_1,y_2)\in\Z^3$.

Diophantine approximation with weights is in a sense an intermediate step between those two settings. Given weights $\pmb\sigma=(\sigma_1,\ldots,\sigma_m)\in\R^m$, $\pmb\rho=(\rho_1,\ldots,\rho_n)\in\R^n$,
\[\sigma_1\geq\ldots\geq\sigma_m>0,\qquad
  \rho_1\geq\ldots\geq\rho_n>0,\qquad
  \sum_{j=1}^m\sigma_j=\sum_{i=1}^n\rho_i=1,\]
the supremum norm is replaced with the weighted norms $|\cdot|_{\pmb\sigma}$ and $|\cdot|_{\pmb\rho}$\,,
\[|\vec x|_{\pmb\sigma}=\max_{1\leq j\leq m}|x_j|^{1/\sigma_j}\qquad\text{ for }\vec x=(x_1,\ldots,x_m),\]
\[|\vec y|_{\pmb\rho}=\max_{1\leq i\leq n}|y_i|^{1/\rho_i}\qquad\,\ \ \text{ for }\vec y=(y_1,\ldots,y_n).\ \]
Respectively, instead of \eqref{eq:trivially_weighted_system}, the system
\begin{equation}\label{eq:system_with_weights}
  \begin{cases}
    |\vec x|_{\pmb\sigma}\leq t \\
    |\Theta\vec x-\vec y|_{\pmb\rho}\leq t^{-\gamma}
  \end{cases}
\end{equation}
is considered. Clearly, when all the $\sigma_j$ are equal to $1/m$ and all the $\rho_i$ are equal to $1/n$, \eqref{eq:system_with_weights} turns into \eqref{eq:trivially_weighted_system}.

\begin{definition}\label{def:ordinary_weighted}
  The \emph{weighted Diophantine exponent} $\omega_{\pmb\sigma,\pmb\rho}(\Theta)$ is defined as the supremum of real $\gamma$ such that the system \eqref{eq:system_with_weights} admits nonzero solutions in $(\vec x,\vec y)\in\Z^{m+n}$ for some arbitrarily large $t$.
\end{definition}

It is well known that, as a rule, there is a relation between problems concerning $\Theta$ and problems concerning the transposed matrix $\tr\Theta$. This relation is provided by the so called transference principle discovered by A.\,Ya.\,Khintchine in \cite{khintchine_palermo}. Recently, in paper \cite{german_M_2020}, the following transference result was obtained. It generalises Dyson's transference theorem \cite{dyson} to the weighted setting.

\begin{theorem}\label{t:weighted_Dyson}
  Set $\omega=\omega_{\pmb\sigma,\pmb\rho}(\Theta)$ and $\tr\omega=\omega_{\pmb\rho,\pmb\sigma}(\tr\Theta)$. Then
  \begin{equation}\label{eq:weighted_Dyson}
    \omega\geq
    \frac{\big(\sigma_m^{-1}-1\big)+\rho_n^{-1}\tr\omega}
         {\sigma_m^{-1}+\big(\rho_n^{-1}-1\big)\tr\omega}\,.
  \end{equation}
\end{theorem}

In Section \ref{sec:weighted_exp} we interpret \eqref{eq:weighted_Dyson} in terms of parametric geometry of numbers and split it into a chain of inequalities for intermediate exponents we define therein.

\subsection{Diophantine exponents of lattices}\label{sec:lattice_exp_intro}

For each $\vec z=(z_1,\ldots,z_d)\in\R^d$ let us set
\[\Pi(\vec z)=\prod_{1\leq i\leq d}|z_i|^{1/d}.\]
Let $\La$ be a full rank lattice in $\R^d$ of covolume $1$. Its \emph{Diophantine exponent} is defined as
\[\omega(\La)=\sup\Big\{\gamma\in\R\ \Big|\,\Pi(\vec z)\leq|\vec z|^{-\gamma}\text{ for infinitely many }\vec z\in\La \Big\},\]
where $|\cdot|$ is again the supremum norm. It follows from Minkowski's convex body theorem that $\omega(\La)$ is nonnegative for every $\La$.

Consider the dual lattice
\[ \La^\ast=\Big\{ \vec w\in\R^d \,\Big|\, \langle\vec w,\vec z\rangle\in\Z\text{ for each }\vec z\in\La \Big\}, \]
where $\langle\,\cdot\,,\,\cdot\,\rangle$ is the inner product.

In this setting the phenomenon of transference can also be observed.
The following result was obtained in \cite{german_2017}.

\begin{theorem} \label{t:lattice_transference}
  The exponents $\omega(\La),\omega(\La^\ast)$ are simultaneously zero. If they are nonzero, then
  \begin{equation} \label{eq:lattice_transference}
    1+\omega(\La)^{-1}\leq
    (d-1)^2(1+\omega(\La^\ast)^{-1}).
  \end{equation}
\end{theorem}

In Section \ref{sec:lattice_exp} we interpret \eqref{eq:lattice_transference} in terms of parametric geometry of numbers and split it into a chain of inequalities for intermediate exponents we define therein.

\vskip 5mm
The rest of the paper is organised as follows. In Section \ref{sec:properties_of_L_and_S} we study local properties of the functions $L_k$ and $S_k$. In Sections \ref{sec:weighted_exp} and \ref{sec:lattice_exp} we apply the results of Section \ref{sec:properties_of_L_and_S} to the theory of Diophantine approximation with weights and to the theory of Diophantine exponents of lattices, respectively. In both cases we define intermediate exponents of two types (Sections \ref{sec:weighted_inter_exp}--\ref{sec:weighted_dual} and \ref{sec:lattice_inter_exp}--\ref{sec:lattice_inter_exp_second_type}) basing on the intuition provided by the parametric approach and by the Schmidt--Summerer exponents that naturally arise within the parametric geometry of numbers. In Sections \ref{sec:application_weighted} and \ref{sec:application_lattice_exp} we perform the announced splitting of Theorems \ref{t:weighted_Dyson} and \ref{t:lattice_transference} respectively. Section \ref{sec:diagram} complements the case of weighted Diophantine approximation with a diagram demonstrating the geometry of the two-dimensional subspace of $\cT$ corresponding to the given weights from the point of view of the transference phenomenon. Finally, in Sections \ref{sec:unit_product} and \ref{sec:pseudo_unit_product}, we prove two series of transference relations for intermediate exponents of the first type. As a corollary of those relations, we obtain in Section \ref{sec:inhomogeneous} a transference theorem for the inhomogeneous weighted Diophantine approximation proved recently by S.\,Chow, A.\,Ghosh et al. \cite{ghosh_marnat_Pisa_2019}.

\section{Local properties of $L_k$ and $S_k$}\label{sec:properties_of_L_and_S}

In this Section we use Minkowski's second theorem and Mahler's theorem on successive minima of compound bodies to derive some important properties of $L_k(\pmb\tau)$ and $S_k(\pmb\tau)$.

For each $\pmb\tau\in\cT$ set
\[|\pmb\tau|_+=\max_{1\leq i\leq d}\tau_i,
  \qquad
  |\pmb\tau|_-=|-\pmb\tau|_+=-\min_{1\leq i\leq d}\tau_i.\]
Clearly,
\[|\pmb\tau|=\max(|\pmb\tau|_-,|\pmb\tau|_+),\]
\begin{equation}\label{eq:tau_minus_vs_tau_plus}
  |\pmb\tau|_+/(d-1)\leq|\pmb\tau|_-\leq(d-1)|\pmb\tau|_+\ .
\end{equation}

\begin{proposition} \label{prop:properties_of_L_k}
  The functions $L_k(\pmb\tau)$ enjoy the following properties:

  \vskip1.5mm
  \textup{(i)} $L_1(\pmb\tau)\leq\ldots\leq L_d(\pmb\tau)$;

  \vskip1.5mm
  \textup{(ii)} $0\leq-L_1(\pmb\tau)\leq|\pmb\tau|_++O(1)$;

  \vskip1.5mm
  \textup{(iii)} $L_d(\pmb\tau)\leq|\pmb\tau|_-+O(1)$;

  \vskip1.5mm
  \textup{(iv)} each $L_k(\pmb\tau)$ is continuous and piecewise linear.
\end{proposition}

\begin{proof}
Statement \textup{(i)} follows immediately from the definition of successive minima. The inequality $L_1(\pmb\tau)\leq0$ is a corollary of Minkowski's convex body theorem. The rest of \textup{(ii)} and \textup{(iii)} is provided by
\[e^{-|\pmb\tau|_+}\cB_{\pmb\tau}\subset\cB\implies
  \lambda_1(e^{-|\pmb\tau|_+}\cB_{\pmb\tau})\geq\lambda_1(\cB)\implies
  \lambda_1(\cB_{\pmb\tau})\geq e^{-|\pmb\tau|_+}\lambda_1(\cB)\]
and
\[e^{|\pmb\tau|_-}\cB_{\pmb\tau}\supset\cB\implies
  \lambda_d(e^{|\pmb\tau|_-}\cB_{\pmb\tau})\leq\lambda_d(\cB)\implies
  \lambda_d(\cB_{\pmb\tau})\leq e^{|\pmb\tau|_-}\lambda_d(\cB).\]
Here, as before, we denote by $\cB$ the cube that is the unit ball in supremum norm.

Let us prove \textup{(iv)}. For each nonzero $\vec z\in\La$ let us denote by $\lambda_{\vec z}(\cB_{\pmb\tau})$ the infimum of positive $\lambda$ such that $\lambda\cB_{\pmb\tau}$ contains $\vec z$, and set
$L_{\vec z}(\pmb\tau)=\log\big(\lambda_{\vec z}(\cB_{\pmb\tau})\big)$.
If $\vec z=(z_1,\ldots,z_d)$, then
\[
  \lambda_{\vec z}(\cB_{\pmb\tau})=\max_{1\leq i\leq d}(|z_i|e^{-\tau_i}),
  \qquad
  L_{\vec z}(\pmb\tau)=\max_{\substack{1\leq i\leq d \\ z_i\neq0}}\big(\log|z_i|-\tau_i\big),
\]
i.e. $L_{\vec z}(\pmb\tau)$ is continuous and piecewise linear. Note that for each $\pmb\tau$ and each $k=1,\ldots,d$ there is a $\vec z=\vec z(\pmb\tau,k)\in\La$ such that $\lambda_k(\cB_{\pmb\tau})=\lambda_{\vec z}(\cB_{\pmb\tau})$. Hence, denoting
\begin{equation} \label{eq:La_k}
  \La_k=\Big\{ \vec z\in\La\, \Big|\, \exists\,\pmb\tau:\,\lambda_k(\cB_{\pmb\tau})=\lambda_{\vec z}(\cB_{\pmb\tau}) \Big\},
\end{equation}
we get
\[L_k(\pmb\tau)=\min_{\vec z\in\La_k}L_{\vec z}(\pmb\tau).\]
Thus, $L_k(\pmb\tau)$ is indeed continuous and piecewise linear.
\end{proof}

\begin{proposition} \label{prop:properties_of_S_k}
  The functions $S_k(\pmb\tau)$ enjoy the following properties:

  \vskip1.5mm
  \textup{(i)} $-\log d!\leq S_d(\pmb\tau)\leq0$;

  \vskip1.5mm
  \textup{(ii)} $\dfrac{k+1}{k}S_k(\pmb\tau)\leq S_{k+1}(\pmb\tau)\leq\dfrac{d-k-1}{d-k}S_k(\pmb\tau)$;

  \vskip1.5mm
  \textup{(iii)} $(d-1)S_1(\pmb\tau)\leq S_{d-1}(\pmb\tau)\leq S_1(\pmb\tau)/(d-1)$;

  \vskip3mm
  \textup{(iv)} $S_{d-1}(\pmb\tau)=-L_d(\pmb\tau)+O(1)$.
\end{proposition}

\begin{proof}
  Statement \textup{(i)} follows from Minkowski's second theorem, which states that
  \[\frac1{d!}\leq\prod_{1\leq k\leq d}\lambda_k(\cB_{\pmb\tau})\leq1.\]

  Furthermore, statement \textup{(i)} of Proposition \ref{prop:properties_of_L_k} and statement \textup{(i)} of the current Proposition imply
  \[S_k(\pmb\tau)\leq kL_{k+1}(\pmb\tau)\]
  and
  \[S_k(\pmb\tau)+(d-k)L_{k+1}(\pmb\tau)\leq S_d(\pmb\tau)\leq0.\]
  Hence
  \[\frac1{k}S_k(\pmb\tau)\leq L_{k+1}(\pmb\tau)\leq\dfrac{-1}{d-k}S_k(\pmb\tau),\]
  and \textup{(ii)} follows.

  Applying statement \textup{(ii)} consequently, we get statement \textup{(iii)}.

  As for statement \textup{(iv)}, it is an immediate corollary of statement \textup{(i)}.
\end{proof}

We remind that $\La^\ast$ denotes the dual lattice.

\begin{proposition} \label{prop:L_and_S_for_dual_lattices}
  For every $\pmb\tau\in\cT$ we have

  \vskip1.5mm
  \textup{(i)} %\label{eq:L_for_dual_lattices}
  \hskip2.5mm
  $-\log d\leq L_k(\La,\pmb\tau)+L_{d+1-k}(\La^\ast,-\pmb\tau)\leq\log d!,\qquad\quad\ k=1,\ldots,d$;

  \vskip1.5mm
  \textup{(ii)} %\label{eq:S_for_dual_lattices}
  $-k\log d\leq S_k(\La,\pmb\tau)-S_{d-k}(\La^\ast,-\pmb\tau)\leq(k+1)\log d!,\quad k=1,\ldots,d-1$.
\end{proposition}

\begin{proof}
  In his paper \cite{mahler_casopis_convex} K.\,Mahler proved that for a parallelepiped $\cB_{\pmb\tau}$ and its polar cross-polytope $\cB_{\pmb\tau}^\circ$ we have
  \begin{equation} \label{eq:mahler_for_cube_and_octahedron}
    1\leq\lambda_k(\cB_{\pmb\tau},\La)\lambda_{d+1-k}(\cB_{\pmb\tau}^\circ,\La^\ast)\leq d!.
  \end{equation}
  Since
  \[\cB_{\pmb\tau}^\circ=
    (D_{\pmb\tau}\cB)^\circ=
    D_{\pmb\tau}^{-1}\cB^\circ=
    D_{-\pmb\tau}\cB^\circ,\]
  we have
  \[d^{-1}\cB_{-\pmb\tau}\subset\cB_{\pmb\tau}^\circ\subset\cB_{-\pmb\tau}.\]
  Therefore, \eqref{eq:mahler_for_cube_and_octahedron} implies
  \[\frac1d\leq\lambda_k(\cB_{\pmb\tau},\La)\lambda_{d+1-k}(\cB_{-\pmb\tau},\La^\ast)\leq d!.\]
  Taking the $\log$ of all sides, we get \textup{(i)}.

  Furthermore, statement \textup{(i)} implies
  \begin{equation} \label{eq:L_for_dual_lattices_summed_up}
    -k\log d\leq S_k(\La,\pmb\tau)+\big(S_d(\La^\ast,-\pmb\tau)-S_{d-k}(\La^\ast,-\pmb\tau)\big)\leq k\log d!,
  \end{equation}
  whereas by statement \textup{(i)} of Proposition \ref{prop:properties_of_S_k}
  \begin{equation} \label{eq:S_d_for_dual_lattice}
    0\leq-S_d(\La^\ast,-\pmb\tau)\leq\log d!.
  \end{equation}
  Summing up \eqref{eq:L_for_dual_lattices_summed_up} and \eqref{eq:S_d_for_dual_lattice}, we get \textup{(ii)}.
\end{proof}

\begin{proposition} \label{prop:essence_of_transference_split_up}
  For every $\pmb\tau\in\cT$ we have

  \textup{(i)}
  $L_k(\La,\pmb\tau)=-L_{d+1-k}(\La^\ast,-\pmb\tau)+O(1),\ \ k=1,\ldots,d$; \vphantom{$\frac{\big|}{|}$}

  \textup{(ii)}
  $S_k(\La,\pmb\tau)=S_{d-k}(\La^\ast,-\pmb\tau)+O(1),\ \ k=1,\ldots,d-1$; \vphantom{$\frac{\big|}{|}$}

  \textup{(iii)}
  $\displaystyle
    S_1(\La,\pmb\tau)\leq\ldots\leq
    \frac{S_k(\La,\pmb\tau)}{k}\leq\ldots\leq
    \frac{S_{d-1}(\La,\pmb\tau)}{d-1}\leq
    \frac{S_{d}(\La,\pmb\tau)}{d}=O(1)$;

  \textup{(iv)}
  $\displaystyle
    \frac{S_1(\La,\pmb\tau)}{d-1}\geq\ldots\geq
    \frac{S_k(\La,\pmb\tau)}{d-k}\geq\ldots\geq
    S_{d-1}(\La,\pmb\tau)$.
  \\
  \vphantom{$\frac{\big|}{}$}
  Here we assume that the implied constants depend only on $d$.
\end{proposition}

\begin{proof}
  Statements \textup{(i)}, \textup{(ii)} follow from Proposition \ref{prop:L_and_S_for_dual_lattices}. Statements \textup{(iii)}, \textup{(iv)} follow from statements \textup{(i)}, \textup{(ii)} of Proposition \ref{prop:properties_of_S_k}.
\end{proof}

\begin{corollary} \label{cor:essence_of_transference}
  For every $\pmb\tau\in\cT$ we have
  \[
    S_1(\La,\pmb\tau)\leq\dfrac{S_1(\La^\ast,-\pmb\tau)}{d-1}+O(1),
  \]
  assuming that the implied constant depends only on $d$.
\end{corollary}

As we shall see in the next two Sections, Corollary \ref{cor:essence_of_transference} is the core of both Theorem \ref{t:weighted_Dyson} and Theorem \ref{t:lattice_transference}. Evidently, statements \textup{(ii)} and \textup{(iii)} of Proposition \ref{prop:essence_of_transference_split_up} split Corollary \ref{cor:essence_of_transference} into a chain of inequalities between the corresponding values of $S_k$.

\section{Weighted exponents and multiparametric geometry of numbers}\label{sec:weighted_exp}

\subsection{Intermediate exponents}\label{sec:weighted_inter_exp}

As in Section \ref{sec:weighted_DA}, let us fix an $n\times m$ real matrix $\Theta$ and weights $\pmb\sigma=(\sigma_1,\ldots,\sigma_m)\in\R^m$, $\pmb\rho=(\rho_1,\ldots,\rho_n)\in\R^n$ such that
\[\sigma_1\geq\ldots\geq\sigma_m>0,\qquad
  \rho_1\geq\ldots\geq\rho_n>0,\qquad
  \sum_{j=1}^m\sigma_j=\sum_{i=1}^n\rho_i=1.\]
Let us supplement $\omega_{\pmb\sigma,\pmb\rho}(\Theta)$ with the following family of intermediate exponents.

\begin{definition}\label{def:intermediate_ordinary_weighted}
  Let $k$ be an integer, $1\leq k\leq d$.
  We define the $k$-th \emph{weighted Diophantine exponent} $\omega_{\pmb\sigma,\pmb\rho}^{(k)}(\Theta)$ as the supremum of real $\gamma$ such that the system
  \eqref{eq:system_with_weights}
  admits $k$ linearly independent solutions in $(\vec x,\vec y)\in\Z^{m+n}$ for some arbitrarily large $t$.
\end{definition}

Clearly, $\omega_{\pmb\sigma,\pmb\rho}(\Theta)=\omega_{\pmb\sigma,\pmb\rho}^{(1)}(\Theta)$.

\begin{definition}\label{def:intermediate_uniform_weighted}
  Let $k$ be an integer, $1\leq k\leq d$.
  We define the $k$-th \emph{weighted uniform Diophantine exponent} $\hat\omega_{\pmb\sigma,\pmb\rho}^{(k)}(\Theta)$ as the supremum of real $\gamma$ such that the system
  \eqref{eq:system_with_weights}
  admits $k$ linearly independent solutions in $(\vec x,\vec y)\in\Z^{m+n}$ for every $t$ large enough.
\end{definition}

For every $k$ we have
\[
  \omega_{\pmb\sigma,\pmb\rho}^{(k)}(\Theta)\geq\hat\omega_{\pmb\sigma,\pmb\rho}^{(k)}(\Theta)\geq0,
\]
since every parallelepiped determined by \eqref{eq:system_with_weights} with $\gamma\leq0$ and $t\geq1$ contains a basis of $\Z^{m+n}$.
Let us interpret the exponents just defined in terms of multiparametric geometry of numbers.

The choice of a lattice is rather standard for problems concerning systems of linear forms, and it does not depend on the weights. We set
\[
  \La=\La(\Theta)=
  \begin{pmatrix}
    \vec I_m & \\
    -\Theta  & \vec I_n
  \end{pmatrix}
  \Z^d,
  \qquad
  d=m+n.
\]
As for the subset of $\cT$ along which $\pmb\tau$ is supposed to tend to infinity, one could say that it should be a one-dimensional subspace, and this would essentially be true. At least, it is literally true when the weights are trivial (i.e. when all the $\sigma_j$ are equal to $1/m$, and all the $\rho_i$ are equal to $1/n$; a detailed description of this case is given in \cite{german_AA_2012}). However, if the weights are nontrivial, there appears a whole family of one-dimensional subspaces, all of which should be taken into account. Let us set
\begin{equation}\label{eq:e1_e2_definition}
\begin{array}{l}
  \vec e_1=\vec e_1(\pmb\sigma,\pmb\rho)=
  (1-d\sigma_1,\ldots,1-d\sigma_m,\underbrace{1,\ldots,1}_{n}), \\
  \vec e_2=\vec e_2(\pmb\sigma,\pmb\rho)=
  (\underbrace{1,\ldots,1}_{m},1-d\rho_n,\ldots,1-d\rho_1).
  \vphantom{\Big|}
\end{array}
\end{equation}
It is clear that in the trivially weighted setting $\vec e_1$ and $\vec e_2$ are proportional, whereas in the case of non-trivial weights they span a two-dimensional subspace of $\cT$. Let us also set for each $\gamma\in\R$
\begin{equation}\label{eq:pmb_mu}
  \pmb\mu=\pmb\mu(\pmb\sigma,\pmb\rho,\gamma)=-\vec e_1+\gamma\vec e_2.
\end{equation}

\begin{definition}\label{def:psi_weighted}
  Given $\La=\La(\Theta)$, $\pmb\mu=\pmb\mu(\pmb\sigma,\pmb\rho,\gamma)$, and $k\in\{1,\ldots,d\}$, the quantities
  \[\bpsi_k(\La,\pmb\mu)=\liminf_{s\to+\infty}\frac{L_k(\La,s\pmb\mu)}{s},
    \qquad
    \apsi_k(\La,\pmb\mu)=\limsup_{s\to+\infty}\frac{L_k(\La,s\pmb\mu)}{s}\]
  are called the \emph{Schmidt--Summerer lower and upper exponents of the first type}, and the quantities
  \[\bPsi_k(\La,\pmb\mu)=\liminf_{s\to+\infty}\frac{S_k(\La,s\pmb\mu)}{s},
    \qquad
    \aPsi_k(\La,\pmb\mu)=\limsup_{s\to+\infty}\frac{S_k(\La,s\pmb\mu)}{s}\]
  are called the \emph{Schmidt--Summerer lower and upper exponents of the second type}.
\end{definition}

\begin{remark}
  It is easily verified that, if $\pmb\mu=(\mu_1,\ldots,\mu_d)=\pmb\mu(\pmb\sigma,\pmb\rho,\gamma)$, $\gamma\neq0$, then $\pmb\sigma$, $\pmb\rho$, and $\gamma$ are uniquely restored from $\pmb\mu$ by the relations
  \[
    \gamma=\frac1m\bigg(\sum_{j=1}^{m}\mu_j-n\bigg),
    \quad
    \sigma_j=\frac1d\big(\mu_j+1-\gamma\big),
    \quad
    \rho_i=\frac1{d\gamma}\big(\gamma-1-\mu_{d+1-i}\big),
  \]
  $j=1,\ldots,m$, $i=1,\ldots,n$.
\end{remark}

Given arbitrary $s,\gamma\in\R$ and $\pmb\mu=(\mu_1,\ldots,\mu_d)=\pmb\mu(\pmb\sigma,\pmb\rho,\gamma)$, we shall consider the parallelepipeds
\begin{align}
  \label{eq:P_definition}
  \cP(s,\gamma)=
  \bigg\{ \vec z=(z_1,\ldots,z_d)\in\R^d \ &
  \bigg|
  \begin{array}{ll}
    |z_j|\leq e^{s\sigma_j}, & j=1,\ldots,m \\
    |z_{d+1-i}|\leq e^{-s\rho_i\gamma}, & \,i=1,\ldots,n
  \end{array}
  \bigg\}, \\
  \label{eq:Q_definition}
  \cQ(s,\gamma)=
  \Big\{ \vec z=(z_1,\ldots,z_d)\in\R^d \ & \Big|\ |z_k|\leq e^{s\mu_k},\ k=1,\ldots,d \Big\}.
  \qquad\quad\ \
\end{align}

\begin{proposition}\label{prop:redefinition_weighted}
  The exponent $\omega_{\pmb\sigma,\pmb\rho}^{(k)}(\Theta)$ (resp. $\hat\omega_{\pmb\sigma,\pmb\rho}^{(k)}(\Theta)$) equals the supremum of $\gamma\in\R$ such that $\cP(s,\gamma)$ contains $k$ linearly independent points of $\La$ for some arbitrarily large $s$ (resp. for every $s$ large enough).

  The exponent $\bpsi_k(\La,\pmb\mu)$ (resp. $\apsi_k(\La,\pmb\mu)$) equals the infimum of $\chi\in\R$ such that $e^{s\chi}\cQ(s,\gamma)$ contains $k$ linearly independent points of $\La$ for some arbitrarily large $s$ (resp. for every $s$ large enough).
\end{proposition}

\begin{proof}
  The statements immediately follow from Definitions \ref{def:intermediate_ordinary_weighted}, \ref{def:intermediate_uniform_weighted}, \ref{def:psi_weighted}.
\end{proof}

\begin{proposition}\label{prop:weighted_in_terms_of_schmimmerer}
  Given $\La=\La(\Theta)$, $\pmb\mu=\pmb\mu(\gamma)=\pmb\mu(\pmb\sigma,\pmb\rho,\gamma)$, and $k\in\{1,\ldots,d\}$, we have
  \begin{equation}\label{eq:weighted_in_terms_of_schmimmerer}
  \begin{array}{l}
    \omega_{\pmb\sigma,\pmb\rho}^{(k)}(\Theta)=\gamma\
    \iff\
    \bpsi_k(\La,\pmb\mu(\gamma))=1-\gamma, \\
    \hat\omega_{\pmb\sigma,\pmb\rho}^{(k)}(\Theta)=\gamma\
    \iff\
    \apsi_k(\La,\pmb\mu(\gamma))=1-\gamma.
    \vphantom{1^{\big|}}
  \end{array}
  \end{equation}
  Besides that,
  $\bpsi_k(\La,\pmb\mu(\gamma))+\gamma$ and $\apsi_k(\La,\pmb\mu(\gamma))+\gamma$
  are increasing as functions of $\gamma$.
\end{proposition}

\begin{proof}
  By \eqref{eq:pmb_mu}
  \begin{equation}\label{eq:mus_explicitly}
    \begin{array}{ll}
      \mu_j=d\sigma_j+\gamma-1, & j=1,\ldots,m, \\
      \mu_{d+1-i}=-d\rho_i\gamma+\gamma-1, & \,i=1,\ldots,n.
    \end{array}
  \end{equation}
  For every $s,\gamma\in\R$ we have by \eqref{eq:mus_explicitly}
  \begin{equation}\label{eq:P_is_Q}
    \cP(ds,\gamma)=e^{s(1-\gamma)}\cQ(s,\gamma).
  \end{equation}
  Hence, in view of Proposition \ref{prop:redefinition_weighted}, \eqref{eq:weighted_in_terms_of_schmimmerer} follows.

  It also follows from \eqref{eq:P_is_Q} that for each $\gamma_1,\gamma_2,\chi\in\R$, $0\leq\gamma_1\leq\gamma_2$, we have
  \begin{equation}\label{eq:Q_is_contained_in_Q}
    e^{s(\chi-\gamma_2)}\cQ(s,\gamma_2)
    \subset
    e^{s(\chi-\gamma_1)}\cQ(s,\gamma_1),
  \end{equation}
  as $\cP(ds,\gamma_2)\subset\cP(ds,\gamma_1)$. Hence the monotonicity of $\bpsi_k(\La,\pmb\mu(\gamma))+\gamma$ and $\apsi_k(\La,\pmb\mu(\gamma))+\gamma$ follows.
\end{proof}

\begin{corollary}\label{cor:weighted_in_terms_of_schmimmerer}
  Within the hypothesis of Proposition \ref{prop:weighted_in_terms_of_schmimmerer}, we have
  \begin{equation*}%\label{eq:weighted_in_terms_of_schmimmerer_inequality}
  \begin{array}{l}
    \omega_{\pmb\sigma,\pmb\rho}^{(k)}(\Theta)\geq\gamma\
    \iff\
    \bpsi_k(\La,\pmb\mu(\gamma))\leq1-\gamma, \\
    \hat\omega_{\pmb\sigma,\pmb\rho}^{(k)}(\Theta)\geq\gamma\
    \iff\
    \apsi_k(\La,\pmb\mu(\gamma))\leq1-\gamma.
    \vphantom{1^{\big|}}
  \end{array}
  \end{equation*}
\end{corollary}

\subsection{Intermediate exponents of the second type}\label{sec:weighted_inter_exp_second_type}

The exponents $\omega_{\pmb\sigma,\pmb\rho}^{(k)}(\Theta)$ and $\hat\omega_{\pmb\sigma,\pmb\rho}^{(k)}(\Theta)$ are analogues in the classical setting of Schmidt--Summerer exponents of the first type. The point of view at all those exponents provided by Proposition \ref{prop:redefinition_weighted} proposes a natural analogue of Schmidt--Summerer exponents of the second type. Before giving the corresponding definition, let us introduce some notation.

Given $k\in\{1,\ldots,d\}$, let us consider the $k$-th external power $\bigwedge^k(\R^d)$. It is a $\binom{d}{k}$-dimensional space. We shall write its elements as $\vec Z=(Z_{i_1,\ldots,i_k})$, assuming that $i_1<\ldots<i_k$ and that $(i_1,\ldots,i_k)$ ranges through the set of all $k$-element subsets of $(1,\ldots,d)$.

Given a parallelepiped
\[
  \cP=\Big\{ \vec z=(z_1,\ldots,z_d)\in\R^d \ \Big|\ |z_i|\leq c_i,\ i=1,\ldots,d \Big\},
\]
let us denote
\begin{equation}\label{eq:P_k_definition}
  \cP^{[k]}=\Big\{ \vec Z=(Z_{i_1,\ldots,i_k})\in{\textstyle\bigwedge^k}(\R^d) \ \Big|\ |Z_{i_1,\ldots,i_k}|\leq\prod_{j=1}^{k} c_{i_j} \Big\}.
\end{equation}

Let also $\cP(s,\gamma)$, $\cQ(s,\gamma)$ be defined by \ref{eq:P_definition}, \ref{eq:Q_definition}, and let $\La=\La(\Theta)$.

\begin{definition}\label{def:intermediate_ordinary_weighted_second_type}
  Let $k$ be an integer, $1\leq k\leq d$.
  We define the $k$-th \emph{weighted Diophantine exponent} $\Omega_{\pmb\sigma,\pmb\rho}^{(k)}(\Theta)$ \emph{of the second type} as the supremum of real $\gamma$ such that $\cP^{[k]}(s,\gamma)$ contains a nonzero element of $\bigwedge^k(\La)$ for some arbitrarily large $s$.
\end{definition}

Clearly, $\Omega_{\pmb\sigma,\pmb\rho}^{(1)}(\Theta)=\omega_{\pmb\sigma,\pmb\rho}^{(1)}(\Theta)=\omega_{\pmb\sigma,\pmb\rho}(\Theta)$.

\begin{definition}\label{def:intermediate_uniform_weighted_second_type}
  Let $k$ be an integer, $1\leq k\leq d$.
  We define the $k$-th \emph{weighted uniform Diophantine exponent} $\hat\Omega_{\pmb\sigma,\pmb\rho}^{(k)}(\Theta)$ \emph{of the second type} as the supremum of real $\gamma$ such that $\cP^{[k]}(s,\gamma)$ contains a nonzero element of $\bigwedge^k(\La)$ for every $s$ large enough.
\end{definition}

Since $\det\big(\bigwedge^k(\La)\big)=1$ and $\vol\big(\cP^{[k]}(s,1)\big)=2^{\binom dk}$, it follows by Minkowski's convex body theorem that
\[
  \Omega_{\pmb\sigma,\pmb\rho}^{(k)}(\Theta)\geq\hat\Omega_{\pmb\sigma,\pmb\rho}^{(k)}(\Theta)\geq1.
\]

\begin{proposition}\label{prop:weighted_in_terms_of_schmimmerer_second_type}
  Given $\La=\La(\Theta)$, $\pmb\mu=\pmb\mu(\gamma)=\pmb\mu(\pmb\sigma,\pmb\rho,\gamma)$, and $k\in\{1,\ldots,d\}$, we have
  \begin{equation}\label{eq:weighted_in_terms_of_schmimmerer_second_type}
  \begin{array}{l}
    \Omega_{\pmb\sigma,\pmb\rho}^{(k)}(\Theta)=\gamma\
    \iff\
    \dfrac{\bPsi_k(\La,\pmb\mu(\gamma))}{k}=1-\gamma, \\
    \hat\Omega_{\pmb\sigma,\pmb\rho}^{(k)}(\Theta)=\gamma\
    \iff\
    \dfrac{\aPsi_k(\La,\pmb\mu(\gamma))}{k}=1-\gamma.
    \vphantom{1^{\big|}}
  \end{array}
  \end{equation}
  Besides that,
  $\dfrac{\bPsi_k(\La,\pmb\mu(\gamma))}{k}+\gamma$ and $\dfrac{\aPsi_k(\La,\pmb\mu(\gamma))}{k}+\gamma$
  are increasing as functions of $\gamma$.
\end{proposition}

\begin{proof}
  It follows from Mahler’s theory of compound bodies that
  \[
    \lambda_1\Big(\cQ^{[k]}(s,\gamma),{\textstyle\bigwedge^k}(\La)\Big)
    \asymp
    \prod_{i=1}^{k}\lambda_i\big(\cQ(s,\gamma),\La\big)
  \]
  with the implied constant depending only on d.

  Thus, the exponent $\bPsi_k(\La,\pmb\mu)$ (resp. $\aPsi_k(\La,\pmb\mu)$) equals the infimum of $\chi\in\R$ such that $e^{s\chi}\cQ^{[k]}(s,\gamma)$ contains a nonzero element of $\bigwedge^k(\La)$ for some arbitrarily large $s$ (resp. for every $s$ large enough).

  The rest of the argument is pretty much the same as in the proof of Proposition \ref{prop:weighted_in_terms_of_schmimmerer}. Instead of \eqref{eq:P_is_Q} we have
  \begin{equation*}%\label{eq:P_k_is_Q_k}
    \cP^{[k]}(ds,\gamma)=e^{sk(1-\gamma)}\cQ^{[k]}(s,\gamma),
  \end{equation*}
  and instead of \eqref{eq:Q_is_contained_in_Q} we have
  \begin{equation*}%\label{eq:Q_k_is_contained_in_Q_k}
    e^{sk(\chi-\gamma_2)}\cQ^{[k]}(s,\gamma_2)
    \subset
    e^{sk(\chi-\gamma_1)}\cQ^{[k]}(s,\gamma_1),
  \end{equation*}
  for every $\gamma_1,\gamma_2,\chi\in\R$, $0\leq\gamma_1\leq\gamma_2$.

  Hence \eqref{eq:weighted_in_terms_of_schmimmerer_second_type} and monotonicity follow.
\end{proof}

\begin{corollary}\label{cor:weighted_in_terms_of_schmimmerer_second_type}
  Within the hypothesis of Proposition \ref{prop:weighted_in_terms_of_schmimmerer_second_type}, we have
  \begin{equation*}%\label{eq:weighted_in_terms_of_schmimmerer_inequality}
  \begin{array}{l}
    \Omega_{\pmb\sigma,\pmb\rho}^{(k)}(\Theta)\geq\gamma\
    \iff\
    \dfrac{\bPsi_k(\La,\pmb\mu(\gamma))}{k}\leq1-\gamma, \\
    \hat\Omega_{\pmb\sigma,\pmb\rho}^{(k)}(\Theta)\geq\gamma\
    \iff\
    \dfrac{\aPsi_k(\La,\pmb\mu(\gamma))}{k}\leq1-\gamma.
  \end{array}
  \end{equation*}
\end{corollary}

\subsection{Dual problem}\label{sec:weighted_dual}

Along with the system \eqref{eq:system_with_weights}, i.e. the system
\begin{equation*}
  \begin{cases}
    |\vec x|_{\pmb\sigma}\leq t \\
    |\Theta\vec x-\vec y|_{\pmb\rho}\leq t^{-\gamma}
  \end{cases},
\end{equation*}
let us consider the dual system
\begin{equation}\label{eq:dual_system_with_weights}
  \begin{cases}
    |\vec y|_{\pmb\rho}\leq t \\
    |\tr\Theta\vec y-\vec x|_{\pmb\sigma}\leq t^{-\delta}
  \end{cases}.
\end{equation}
Given $k\in\{1,\ldots,d\}$, the exponent $\omega^{(k)}_{\pmb\rho,\pmb\sigma}(\tr\Theta)$ (resp. $\hat\omega^{(k)}_{\pmb\rho,\pmb\sigma}(\tr\Theta)$) is defined by Definition \ref{def:intermediate_ordinary_weighted} (resp. \ref{def:intermediate_uniform_weighted}) as the supremum of real $\delta$ such that the system \eqref{eq:dual_system_with_weights} admits $k$ linearly independent solutions in $(\vec x,\vec y)\in\Z^{m+n}$ for some arbitrarily large $t$ (resp. for every $t$ large enough).

The previous Section already provides an interpretation of the exponents $\omega^{(k)}_{\pmb\rho,\pmb\sigma}(\tr\Theta)$, $\hat\omega^{(k)}_{\pmb\rho,\pmb\sigma}(\tr\Theta)$. It suffices to swap the triple $(\Theta,\pmb\sigma,\pmb\rho)$ for $(\tr\Theta,\pmb\rho,\pmb\sigma)$.
%thus swapping the system \eqref{eq:system_with_weights} for \eqref{eq:dual_system_with_weights}.
However, in the context of the transference phenomenon, it is very useful to give another interpretation involving the dual lattice for $\La(\Theta)$ and preserving the subspace generated by $\vec e_1$ and $\vec e_2$. So, let us consider the dual lattice
\[
  \La^\ast=\La^\ast(\Theta)=
  \begin{pmatrix}
    \vec I_m & \tr\Theta \\
    & \vec I_n
  \end{pmatrix}
  \Z^d
\]
and set for each $\delta\in\R$
\begin{equation}\label{eq:pmb_mu_ast}
  \pmb\mu^\ast=\pmb\mu^\ast(\pmb\sigma,\pmb\rho,\delta)=\delta\vec e_1-\vec e_2,
\end{equation}
where $\vec e_1$ and $\vec e_2$ are defined by \eqref{eq:e1_e2_definition}.

\begin{proposition}\label{prop:dual_weighted_in_terms_of_schmimmerer}
  Given $\La^\ast=\La^\ast(\Theta)$, $\pmb\mu^\ast=\pmb\mu^\ast(\delta)=\pmb\mu^\ast(\pmb\sigma,\pmb\rho,\delta)$, and $k\in\{1,\ldots,d\}$, we have
  \begin{equation}\label{eq:dual_weighted_in_terms_of_schmimmerer}
  \begin{array}{l}
    \omega_{\pmb\rho,\pmb\sigma}^{(k)}(\tr\Theta)=\delta\
    \iff\
    \bpsi_k(\La^\ast,\pmb\mu^\ast(\delta))=1-\delta, \\
    \hat\omega_{\pmb\rho,\pmb\sigma}^{(k)}(\tr\Theta)=\delta\
    \iff\
    \apsi_k(\La^\ast,\pmb\mu^\ast(\delta))=1-\delta.
    \vphantom{1^{\big|}}
  \end{array}
  \end{equation}
  Besides that,
  $\bpsi_k(\La^\ast,\pmb\mu^\ast(\delta))+\delta$ and $\apsi_k(\La^\ast,\pmb\mu^\ast(\delta))+\delta$
  are increasing as functions of $\delta$.
\end{proposition}

\begin{proof}
  The proof repeats, mutatis mutandis, that of Proposition \ref{prop:weighted_in_terms_of_schmimmerer}. We should replace $\pmb\mu(\gamma)$ with $\pmb\mu^\ast(\delta)$,
  \begin{equation*}%\label{eq:mus_ast_explicitly}
    \begin{array}{ll}
      \mu^\ast_j=-d\sigma_j\delta+\delta-1, & j=1,\ldots,m, \\
      \mu^\ast_{d+1-i}=d\rho_i+\delta-1, & \,i=1,\ldots,n,
    \end{array}
  \end{equation*}
  and consider, instead of $\cP(s,\gamma)$, $\cQ(s,\gamma)$, the parallelepipeds
  \[
    \cP^\ast(s,\delta)=
    \bigg\{ \vec z=(z_1,\ldots,z_d)\in\R^d \
    \bigg|
    \begin{array}{ll}
      |z_j|\leq e^{-s\sigma_j\delta}, & j=1,\ldots,m \\
      |z_{d+1-i}|\leq e^{s\rho_i}, & \,i=1,\ldots,n
    \end{array}
    \bigg\},
  \]
  \[
    \cQ^\ast(s,\delta)=
    \Big\{ \vec z=(z_1,\ldots,z_d)\in\R^d \ \Big|\ |z_k|\leq e^{s\mu^\ast_k},\ k=1,\ldots,d \Big\}.
    \qquad\ \
  \]
  Then, same as \eqref{eq:P_is_Q} and \eqref{eq:Q_is_contained_in_Q}, we get
  \begin{equation*}%\label{eq:P_is_Q}
    \cP^\ast(ds,\delta)=e^{s(1-\delta)}\cQ^\ast(s,\delta)
  \end{equation*}
  and
  \begin{equation*}%\label{eq:Q_is_contained_in_Q}
    e^{s(\chi-\delta_2)}\cQ^\ast(s,\delta_2)
    \subset
    e^{s(\chi-\delta_1)}\cQ^\ast(s,\delta_1),
  \end{equation*}
  for each $\delta_1,\delta_2,\chi\in\R$, $0\leq\delta_1\leq\delta_2$.
  Hence \eqref{eq:dual_weighted_in_terms_of_schmimmerer} and monotonicity follow.
\end{proof}

\begin{proposition}\label{prop:dual_weighted_in_terms_of_schmimmerer_second_type}
  Given $\La^\ast=\La^\ast(\Theta)$, $\pmb\mu^\ast=\pmb\mu^\ast(\delta)=\pmb\mu^\ast(\pmb\sigma,\pmb\rho,\delta)$, and $k\in\{1,\ldots,d\}$, we have
  \begin{equation}\label{eq:dual_weighted_in_terms_of_schmimmerer_second_type}
  \begin{array}{l}
    \Omega_{\pmb\rho,\pmb\sigma}^{(k)}(\tr\Theta)=\delta\
    \iff\
    \dfrac{\bPsi_k(\La^\ast,\pmb\mu^\ast(\delta))}{k}=1-\delta, \\
    \hat\Omega_{\pmb\rho,\pmb\sigma}^{(k)}(\tr\Theta)=\delta\
    \iff\
    \dfrac{\aPsi_k(\La^\ast,\pmb\mu^\ast(\delta))}{k}=1-\delta.
    \vphantom{1^{\big|}}
  \end{array}
  \end{equation}
  Besides that,
  $\dfrac{\bPsi_k(\La^\ast,\pmb\mu^\ast(\delta))}{k}+\delta$ and $\dfrac{\aPsi_k(\La^\ast,\pmb\mu^\ast(\delta))}{k}+\delta$
  are increasing as functions of $\delta$.
\end{proposition}

The proof of Proposition \ref{prop:dual_weighted_in_terms_of_schmimmerer_second_type} is obtained by changing the proof of Proposition \ref{prop:weighted_in_terms_of_schmimmerer_second_type} in the very same way we obtained the proof of Proposition \ref{prop:dual_weighted_in_terms_of_schmimmerer} from that of Proposition \ref{prop:weighted_in_terms_of_schmimmerer}.

\subsection{Application of general theory and splitting Theorem \ref{t:weighted_Dyson}}\label{sec:application_weighted}

Let us apply Proposition \ref{prop:essence_of_transference_split_up} and Corollary \ref{cor:essence_of_transference} for $\La=\La(\Theta)$, $\pmb\tau=s\pmb\mu$. Dividing every relation thus obtained by $s$ and sending $s$ to $+\infty$, we get the following statements on the Schmidt--Summerer exponents.

\begin{proposition} \label{prop:essence_of_transference_split_up_weighted}
  Given an arbitrary $\gamma\in\R$, let $\La=\La(\Theta)$, $\pmb\mu=\pmb\mu(\pmb\sigma,\pmb\rho,\gamma)$. Then

  \textup{(i)}
  $\bpsi_k(\La,\pmb\mu)=-\apsi_{d+1-k}(\La^\ast,-\pmb\mu)$, \vphantom{$\frac{\big|}{|}$}

  \phantom{\textup{(i)}}
  $\apsi_k(\La,\pmb\mu)=-\bpsi_{d+1-k}(\La^\ast,-\pmb\mu);\ \ k=1,\ldots,d$; \vphantom{$\Big|$}

  \textup{(ii)}
  $\bPsi_k(\La,\pmb\mu)=\bPsi_{d-k}(\La^\ast,-\pmb\mu)$, \vphantom{$\frac{\big|}{|}$}

  \phantom{\textup{(ii)}}
  $\aPsi_k(\La,\pmb\mu)=\aPsi_{d-k}(\La^\ast,-\pmb\mu),\ \ k=1,\ldots,d-1$; \vphantom{$\Big|$}

  \textup{(iii)}
  $
    \displaystyle
    \bPsi_1(\La,\pmb\mu)
    \leq\ldots\leq
    \frac{\bPsi_k(\La,\pmb\mu)}{k}
    \leq\ldots\leq
    \frac{\bPsi_{d-1}(\La,\pmb\mu)}{d-1}
    \leq
    \frac{\bPsi_d(\La,\pmb\mu)}{d}
    =0
  $,

  \phantom{\textup{(iii)}}
  $
    \displaystyle
    \aPsi_1(\La,\pmb\mu)
    \leq\ldots\leq
    \frac{\aPsi_k(\La,\pmb\mu)}{k}
    \leq\ldots\leq
    \frac{\aPsi_{d-1}(\La,\pmb\mu)}{d-1}
    \leq
    \frac{\aPsi_d(\La,\pmb\mu)}{d}
    =0
  $;
  \vphantom{$\bigg|$}

  \textup{(iv)}
  $
    \displaystyle
    \frac{\bPsi_1(\La,\pmb\mu)}{d-1}
    \geq\ldots\geq
    \frac{\bPsi_k(\La,\pmb\mu)}{d-k}
    \geq\ldots\geq
    \bPsi_{d-1}(\La,\pmb\mu)
  $,
  \vphantom{$\bigg|$}

  \phantom{\textup{(iv)}}
  $
    \displaystyle
    \frac{\aPsi_1(\La,\pmb\mu)}{d-1}
    \geq\ldots\geq
    \frac{\aPsi_k(\La,\pmb\mu)}{d-k}
    \geq\ldots\geq
    \aPsi_{d-1}(\La,\pmb\mu)
  $.
\end{proposition}

\begin{corollary} \label{cor:essence_of_transference_weighted}
  $
    \bPsi_1(\La,\pmb\mu)
    \leq
    \dfrac{\bPsi_1(\La^\ast,-\pmb\mu)}{d-1}
  $.
\end{corollary}

We claim that Corollary \ref{cor:essence_of_transference_weighted} implies Theorem \ref{t:weighted_Dyson}. To show this, let us set for each $\delta\geq0$
\begin{equation}\label{eq:gamma_of_delta}
  \gamma_\delta=
  \begin{cases}
    \dfrac{\big(\sigma_m^{-1}-1\big)+\rho_n^{-1}\delta}
          {\sigma_m^{-1}+\big(\rho_n^{-1}-1\big)\delta}\,, & \mbox{if } \delta\geq1, \\
    \dfrac{\big(\sigma_1^{-1}-1\big)+\rho_1^{-1}\delta\vphantom{1^{\big|}}}
          {\sigma_1^{-1}+\big(\rho_1^{-1}-1\big)\delta}\,, & \mbox{if } \delta\leq1.
  \end{cases}
\end{equation}
It is easy to see that the function $\delta\mapsto\gamma_\delta$ monotonously maps $[0,+\infty]$ onto the segment $[1-\sigma_1,(1-\rho_n)^{-1}]$, and that $\delta=1$ if and only if $\gamma_\delta=1$ (see Figure \ref{fig:gamma_delta}). We naturally assume that $(1-\rho_n)^{-1}=+\infty$ in case $\rho_n=1$ (which holds if and only if $n=1$).

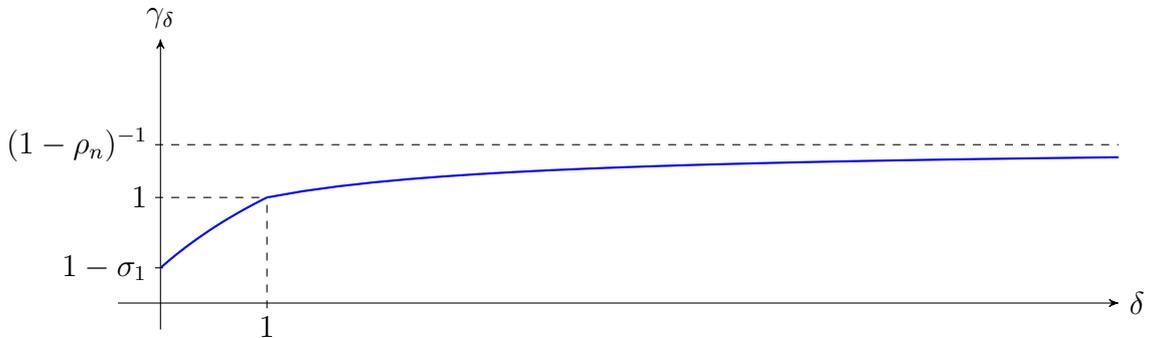
\begin{figure}[h]
\newcommand{\sigmamax}{(2/3)}
\newcommand{\sigmamin}{(1/3)}
\newcommand{\rhomax}{(2/3)}
\newcommand{\rhomin}{(1/3)}
\centering
\begin{tikzpicture}[scale=1.4]
  \draw[->,>=stealth'] (-0.4,0) -- (9,0) node[right] {$\delta$};
  \draw[->,>=stealth'] (0,-0.25) -- (0,2.5) node[above] {$\gamma_\delta$};

  \draw[color=blue,thick] plot[domain=0:1] (\x, {(\sigmamax^(-1)-1+\rhomax^(-1)*\x)/(\sigmamax^(-1)+(\rhomax^(-1)-1)*\x)});
  \draw[color=blue,thick] plot[domain=1:9] (\x, {(\sigmamin^(-1)-1+\rhomin^(-1)*\x)/(\sigmamin^(-1)+(\rhomin^(-1)-1)*\x)});

  \draw[color=black,dashed] (1,-0.05) -- (1,1);
  \draw[color=black,dashed] (-0.05,1) -- (1,1);
  \draw[color=black,dashed] (-0.05,{(1-\rhomin)^(-1)}) -- (9,{(1-\rhomin)^(-1)});
  \draw[color=black,dashed] (-0.05,{1-\sigmamax}) -- (0.05,{1-\sigmamax});

  \draw (1,-0.02) node[below]{$1$};
  \draw (-0.02,1) node[left]{$1$};
  \draw (-0.02,{(1-\rhomin)^(-1)}) node[left]{$(1-\rho_n)^{-1}$};
  \draw (-0.02,{1-\sigmamax}) node[left]{$1-\sigma_1$};
\end{tikzpicture}
\caption{Graph of $\gamma_\delta$} \label{fig:gamma_delta}
\end{figure}

The inverse of \eqref{eq:gamma_of_delta} is
\begin{equation}\label{eq:delta_of_gamma}
  \hskip5.6mm
  \delta=
  \begin{cases}
    \dfrac{\big(1-\sigma_m^{-1}\big)+\sigma_m^{-1}\gamma_\delta}
          {\rho_n^{-1}+\big(1-\rho_n^{-1}\big)\gamma_\delta}\,, & \mbox{if } \gamma_\delta\geq1, \\
    \dfrac{\big(1-\sigma_1^{-1}\big)+\sigma_1^{-1}\gamma_\delta\vphantom{1^{\big|}}}
          {\rho_1^{-1}+\big(1-\rho_1^{-1}\big)\gamma_\delta}\,, & \mbox{if } \gamma_\delta\leq1.
  \end{cases}
\end{equation}

\begin{lemma}\label{l:immersion}
  For each $k\in\{1,\ldots,d\}$ and every $\delta\geq0$ we have
  \begin{equation*}%\label{eq:weighted_in_terms_of_schmimmerer_inequality}
  \begin{array}{l}
    \bpsi_k\big(\La^\ast,\pmb\mu^\ast(\delta)\big)
    \leq 1-\delta
    \implies
    \bpsi_k\big(\La^\ast,-\pmb\mu(\gamma_\delta)\big)
    \leq (d-1)(1-\gamma_\delta), \\
    \apsi_k\big(\La^\ast,\pmb\mu^\ast(\delta)\big)
    \leq 1-\delta
    \implies
    \apsi_k\big(\La^\ast,-\pmb\mu(\gamma_\delta)\big)
    \leq (d-1)(1-\gamma_\delta).
    \vphantom{1^{\big|}}
  \end{array}
  \end{equation*}
\end{lemma}

\begin{proof}
  For each $s\geq0$ and $\delta\geq0$ let us set
  \[
    s_\delta=
    \begin{cases}
      s\big(\rho_n^{-1}+\big(1-\rho_n^{-1}\big)\gamma_\delta\big), & \mbox{if } \delta\geq1, \\
      s\big(\rho_1^{-1}+\big(1-\rho_1^{-1}\big)\gamma_\delta\big), & \mbox{if } \delta\leq1.
    \end{cases}
  \]
  Then
  \[
    s_\delta\delta=
    \begin{cases}
      s\big(\big(1-\sigma_m^{-1}\big)+\sigma_m^{-1}\gamma_\delta\big), & \mbox{if } \delta\geq1, \\
      s\big(\big(1-\sigma_1^{-1}\big)+\sigma_1^{-1}\gamma_\delta\big), & \mbox{if } \delta\leq1,
    \end{cases}
  \]
  and we have for each $j\in\{1,\ldots,m\}$, $i\in\{1,\ldots,n\}$
  \begin{equation}\label{eq:s_delta}
    \begin{aligned}
      s_\delta\delta
      & \geq
      s((1-\sigma_j^{-1})+\sigma_j^{-1}\gamma_\delta), \\
      s_\delta
      & \leq
      s(\rho_i^{-1}+(1-\rho_i^{-1})\gamma_\delta).
    \end{aligned}
  \end{equation}
  Upon some minor calculations, we conclude from \eqref{eq:s_delta} that
  \begin{equation}\label{eq:Q_ast_is_contained_in_Q-}
    e^{s_\delta(1-\delta)}\cQ^\ast(s_\delta,\delta)
    \subset
    e^{s(d-1)(1-\gamma_\delta)}\cQ(-s,\gamma_\delta),
  \end{equation}
  where $\cQ$ and $\cQ^\ast$ are as in the proofs of Propositions \ref{prop:weighted_in_terms_of_schmimmerer} and \ref{prop:dual_weighted_in_terms_of_schmimmerer}. Hence the desired statement follows immediately.
\end{proof}

%  \[
%    \cP^\ast(s,\delta)=
%    \bigg\{ \vec z=(z_1,\ldots,z_d)\in\R \
%    \bigg|
%    \begin{array}{ll}
%      |z_j|\leq e^{-s\sigma_j\delta}, & j=1,\ldots,m \\
%      |z_{d+1-i}|\leq e^{s\rho_i}, & \,i=1,\ldots,n
%    \end{array}
%    \bigg\},
%  \]
%  \begin{equation*}%\label{eq:mus_explicitly}
%    \begin{array}{ll}
%      \mu_j=d\sigma_j+\gamma-1, & \mu^\ast_j=-d\sigma_j\delta+\delta-1, \\
%      \mu_{d+1-i}=-d\rho_i\gamma+\gamma-1, & \mu^\ast_{d+1-i}=d\rho_i+\delta-1.
%    \end{array}
%  \end{equation*}
%  \[
%    \cQ(-s,\gamma)=
%    \Big\{ \vec z=(z_1,\ldots,z_d)\in\R \ \Big|\ |z_k|\leq e^{-s\mu_k},\ k=1,\ldots,d \Big\}.
%  \]
%  \[
%    \cQ^\ast(s,\delta)=
%    \Big\{ \vec z=(z_1,\ldots,z_d)\in\R \ \Big|\ |z_k|\leq e^{s\mu^\ast_k},\ k=1,\ldots,d \Big\}.
%    \
%  \]
%  \[
%    \cP^\ast(ds,\delta)=
%    e^{s(1-\delta)}\cQ^\ast(s,\delta):\
%    \begin{aligned}
%      |z_j|\leq e^{s(-d\sigma_j\delta+\delta-1)+s(1-\delta)} & =e^{ds(-\sigma_j\delta)} \\
%      |z_{d+1-i}|\leq e^{s(d\rho_i+\delta-1)+s(1-\delta)} & =e^{ds\rho_i}
%    \end{aligned}
%  \]
%  \[
%    \cP^\ast(ds_\delta,\delta)\subset
%    e^{s(d-1)(1-\gamma_\delta)}\cQ^-(s,\gamma_\delta):\
%    \begin{aligned}
%      |z_j|\leq
%%      e^{-s(d\sigma_j+\gamma-1)+s(d-1)(1-\gamma)}=
%      & e^{ds(-\sigma_j+1-\gamma)}\geq
%      e^{-ds_\delta\sigma_j\delta} \\
%      |z_{d+1-i}|\leq
%%      e^{-s(-d\rho_i\gamma+\gamma-1)+s(d-1)(1-\gamma)}=
%      & e^{ds(\rho_i\gamma+1-\gamma)}\geq
%      e^{ds_\delta\rho_i}
%    \end{aligned}
%  \]

\begin{theorem}\label{t:weighted_dyson_schmimmerered}
  For every $\delta\geq1$ we have
  \[
    \bpsi_1\big(\La^\ast,\pmb\mu^\ast(\delta)\big)\leq 1-\delta
    \implies
    \bpsi_1\big(\La,\pmb\mu(\gamma_\delta)\big)\leq 1-\gamma_\delta\,.
  \]
\end{theorem}

\begin{proof}
  It suffices to apply Lemma \ref{l:immersion} with $k=1$, Corollary \ref{cor:essence_of_transference_weighted}, and the fact that the exponents $\bpsi_1$ and $\bPsi_1$ are the same.
\end{proof}

Theorem \ref{t:weighted_dyson_schmimmerered}, in view of Corollary \ref{cor:weighted_in_terms_of_schmimmerer}, is a reformulation of Theorem \ref{t:weighted_Dyson}. Note that the condition $\delta\geq1$ reflects the fact that, by Minkowski's convex body theorem, we always have $\omega_{\pmb\sigma,\pmb\rho}(\Theta)\geq1$, in contrast to $\omega_{\pmb\sigma,\pmb\rho}^{(k)}(\Theta)$ with $k\geq2$, which can attain values in the interval $[0,1)$.

The key ingredient in the proof of Theorem \ref{t:weighted_dyson_schmimmerered} is Corollary \ref{cor:essence_of_transference_weighted}. Lemma \ref{l:immersion} can be considered as a rather important but technical statement: its essence is shown by the inclusion \eqref{eq:Q_ast_is_contained_in_Q-} performing a proper rescaling. In its turn, Corollary \ref{cor:essence_of_transference_weighted} gets split by statements \textup{(ii)} and \textup{(iii)} of Proposition \ref{prop:essence_of_transference_split_up_weighted} into a sequence of inequalities between the Schmidt--Summerer exponents $\bPsi_k(\La,\pmb\mu)$. Applying Proposition \ref{prop:weighted_in_terms_of_schmimmerer_second_type}, Corollary \ref{cor:weighted_in_terms_of_schmimmerer_second_type}, statement \textup{(iii)} of Proposition \ref{prop:essence_of_transference_split_up_weighted}, and, of course, Lemma \ref{l:immersion}, we get the following splitting of Theorem \ref{t:weighted_Dyson}.

\begin{theorem}
  Set $\Omega_k=\Omega_{\pmb\sigma,\pmb\rho}^{(k)}(\Theta)$ and $\tr\Omega_k=\Omega_{\pmb\rho,\pmb\sigma}^{(k)}(\tr\Theta)$ for each $k=1,\ldots,d$. Then
  \[
    \Omega_1
    \geq\ldots\geq
    \Omega_k
    \geq\ldots\geq
    \Omega_{d-1}
    \geq
    \frac{\big(\sigma_m^{-1}-1\big)+\rho_n^{-1}\tr\Omega_1}
         {\sigma_m^{-1}+\big(\rho_n^{-1}-1\big)\tr\Omega_1}\,.
  \]
\end{theorem}

\subsection{Transference diagram}\label{sec:diagram}

Application of Lemma \ref{l:immersion} and Corollary \ref{cor:essence_of_transference_weighted} leads us from $\pmb\mu^\ast(\delta)$ to $\pmb\mu(\gamma_\delta)$ via $-\pmb\mu(\gamma_\delta)$. Those three points belong to the subspace spanned by $\vec e_1$, $\vec e_2$. Lemma \ref{l:transference_diagram} and Figure \ref{fig:transference_diagram} below demonstrate how they are related.

\begin{lemma}\label{l:transference_diagram}
  Suppose the weights are nontrivial (i.e. either $\sigma_1\neq\sigma_m$, or $\rho_1\neq\rho_n$).

  If $\delta=1$, then $-\pmb\mu(\gamma_\delta)=\pmb\mu^\ast(\delta)$.

  If $n>1$ (equivalently, $\rho_1<1$, $\rho_n<1$) and $\delta\neq1$, then
  the line through the points $\pmb\mu^\ast(\delta)$ and $-\pmb\mu(\gamma_\delta)$ passes through a point $\pmb\nu$, that depends only on the sign of $\delta-1$. This point can be expressed explicitly as
  \[
    \pmb\nu=
    \begin{cases}
      \dfrac{1}{1-\rho_n^{-1}}
      \big(
        \sigma_m^{-1}\vec e_1+
        \rho_n^{-1}\vec e_2
      \big), & \mbox{if } \delta>1, \\
      \dfrac{1\vphantom{1^{\big|}}}{1-\rho_1^{-1}}
      \big(
        \sigma_1^{-1}\vec e_1+
        \rho_1^{-1}\vec e_2
      \big), & \mbox{if } \delta<1.
    \end{cases}
  \]

  If $n=1$ (equivalently, $\rho_1=\rho_n=1$) and $\delta\neq1$, then the line through the points $\pmb\mu^\ast(\delta)$ and $-\pmb\mu(\gamma_\delta)$ is parallel to either $\sigma_m^{-1}\vec e_1+\vec e_2$, or $\sigma_1^{-1}\vec e_1+\vec e_2$, according to whether $\delta>1$, or $\delta<1$.
\end{lemma}

The proof is elementary and we leave it to the reader. The case  $n>1$, $\delta>1$ is illustrated by Figure \ref{fig:transference_diagram}.

\begin{figure}[h]
\centering
\begin{tikzpicture}[scale=2.1]
  \draw[->,>=stealth'] (-1.9,0) -- (1,0) node[above right] {$\vec e_1$};
  \draw (1,0) -- (5.3,0); % node[right] {$\delta$};
  \draw[->,>=stealth'] (0,-2.05) -- (0,1) node[above right] {$\vec e_2$};
  \draw (0,1) -- (0,2);

  \draw[dashed] (0,-1) -- (5.3,-1);
  \draw[dashed] (0,1) -- (-1.9,1);
  \draw[dashed] (1,0) -- (1,-2.05);
  \draw[dashed] (-1,0) -- (-1,2);
  \draw[dashed] (-1.5,-1.5) -- (1.18,-1.5);
  \draw[dashed] (1.65,-1.5) -- (5.3,-1.5);

  \draw[color=blue,thick] (4.5,-1) -- (-1.5,-1.5);
  \draw[color=blue,thick] (1,-3/2+5/24) -- (-1,3/2-5/24);

  \node[fill=white,draw=blue,thick,circle,inner sep=1.1pt] at (-1.5,-1.5) {};
  \node[fill=blue,draw=blue,circle,inner sep=1.2pt] at (1,-3/2+5/24) {};
  \node[fill=blue,draw=blue,circle,inner sep=1.2pt] at (4.5,-1) {};
  \node[fill=blue,draw=blue,circle,inner sep=1.2pt] at (-1,3/2-5/24) {};

  \draw (-1,3/2-5/24) node[above left]{$\pmb\mu(\gamma_\delta)$};
  \draw (1-0.02,-3/2+5/24+0.06) node[below right]{$-\pmb\mu(\gamma_\delta)$};
  \draw (4.48,-1) node[below right]{$\pmb\mu^\ast(\delta)$};
  \draw (-1.5,-1.5) node[above left]{$\pmb\nu$};
\end{tikzpicture}
\caption{Transference diagram (case $\delta>1$, $n>1$, nontrivial weights)} \label{fig:transference_diagram}
\end{figure}
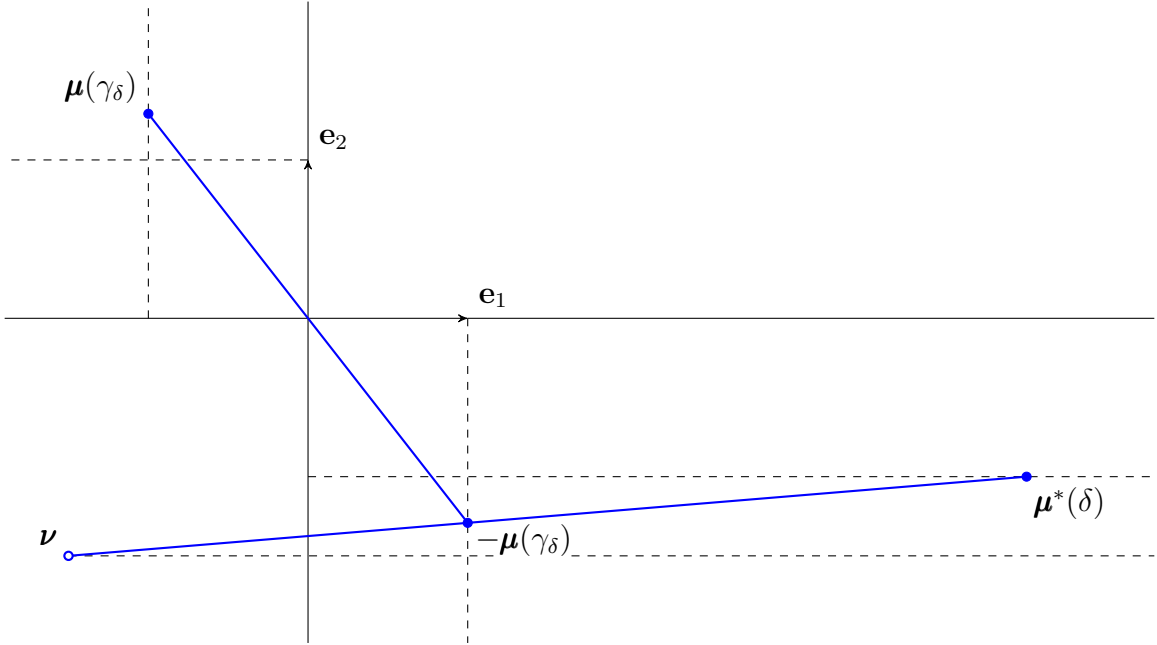

\subsection{Transference equalities for intermediate weighted Diophantine exponents}\label{sec:unit_product}

One might ask whether there is a chain of inequalities between the Schmidt--Summerer exponents of the first type that splits Theorem \ref{t:weighted_Dyson} in a way similar to that described above. Such a result would split Theorem \ref{t:weighted_Dyson} into a chain of inequalities between the exponents $\omega_{\pmb\sigma,\pmb\rho}^{(k)}(\Theta)$. We do not know the answer, though we doubt that such a chain of inequalities exists.

However, there is a relation, and a very nice one, between the exponents $\omega_{\pmb\sigma,\pmb\rho}^{(k)}(\Theta)$ and the exponents $\hat\omega_{\pmb\rho,\pmb\sigma}^{(k)}(\tr\Theta)$. It generalises the corresponding relation that holds in the case of trivial weights (see \cite[Corollary 8.5]{german_AA_2012}). It is given by Theorem \ref{t:omega_times_omega} below.

\begin{lemma}\label{l:for_omega_times_omega}
  For each $k\in\{1,\ldots,d\}$ and every $\gamma>0$ we have
  \begin{equation*}%\label{eq:for_omega_times_omega}
  \begin{array}{l}
    \bpsi_k\big(\La^\ast,-\pmb\mu(\gamma)\big)=
    \gamma\cdot
    \bpsi_k\big(\La^\ast,\pmb\mu^\ast\big(\gamma^{-1}\big)\big), \\
    \apsi_k\big(\La^\ast,-\pmb\mu(\gamma)\big)=
    \gamma\cdot
    \apsi_k\big(\La^\ast,\pmb\mu^\ast\big(\gamma^{-1}\big)\big).
    \vphantom{1^{\big|}}
  \end{array}
  \end{equation*}
\end{lemma}

\begin{proof}
  By \eqref{eq:pmb_mu} and \eqref{eq:pmb_mu_ast} we have $-\pmb\mu(\gamma)=\vec e_1-\gamma\vec e_2=\gamma\pmb\mu^\ast(\gamma^{-1})$. Therefore,
  \[
    \frac{L_k(\La^\ast,-s\pmb\mu(\gamma))}{s}=
    \frac{L_k(\La^\ast,s\gamma\pmb\mu^\ast(\gamma^{-1}))}{s}=
    \gamma\cdot\frac{L_k(\La^\ast,s'\pmb\mu^\ast(\gamma^{-1}))}{s'}\,,
  \]
  where $s'=s\gamma$. Hence the desired relations follow.
\end{proof}

%  \[\bpsi_k(\La,\pmb\mu)=\liminf_{s\to+\infty}\frac{L_k(s\pmb\mu,\La)}{s},
%    \qquad
%    \apsi_k(\La,\pmb\mu)=\limsup_{s\to+\infty}\frac{L_k(s\pmb\mu,\La)}{s}\]
%  \begin{equation}%\label{eq:mus_explicitly}
%    \begin{array}{ll}
%      \mu_j(\gamma)=d\sigma_j+\gamma-1, & j=1,\ldots,m, \\
%      \mu_{d+1-i}(\gamma)=-d\rho_i\gamma+\gamma-1, & \,i=1,\ldots,n.
%    \end{array}
%  \end{equation}
%  \begin{equation}%\label{eq:mus_explicitly}
%    \begin{array}{ll}
%      -\gamma^{-1}\mu_j(\gamma)=-d\sigma_j\gamma^{-1}+\gamma^{-1}-1, & j=1,\ldots,m, \\
%      -\gamma^{-1}\mu_{d+1-i}(\gamma)=d\rho_i+\gamma^{-1}-1, & \,i=1,\ldots,n.
%    \end{array}
%  \end{equation}
%  \begin{equation*}%\label{eq:mus_ast_explicitly}
%    \begin{array}{ll}
%      \mu^\ast_j(\delta)=-d\sigma_j\delta+\delta-1, & j=1,\ldots,m, \\
%      \mu^\ast_{d+1-i}(\delta)=d\rho_i+\delta-1, & \,i=1,\ldots,n,
%    \end{array}
%  \end{equation*}

\begin{theorem}\label{t:omega_times_omega}
  For every $k\in\{1,\ldots,d\}$ we have
  \begin{equation}\label{eq:omega_times_omega}
    \omega_{\pmb\sigma,\pmb\rho}^{(k)}(\Theta)\cdot\hat\omega_{\pmb\rho,\pmb\sigma}^{(d+1-k)}(\tr\Theta)=1,
    \qquad
    \hat\omega_{\pmb\sigma,\pmb\rho}^{(k)}(\Theta)\cdot\omega_{\pmb\rho,\pmb\sigma}^{(d+1-k)}(\tr\Theta)=1.
  \end{equation}
  Here it is assumed that if any of the factors is zero, then the other one is equal to $+\infty$, and vice versa.
\end{theorem}

\begin{proof}
  Proving the first equality in \eqref{eq:omega_times_omega} will suffice, as each of the two equalities can be turned into the other by simple swapping the tuple $(\Theta,\pmb\sigma,\pmb\rho,k)$ for $(\tr\Theta,\pmb\rho,\pmb\sigma,d+1-k)$.

  Suppose $\gamma>0$.
  Applying successively Proposition \ref{prop:weighted_in_terms_of_schmimmerer}, statement \textup{(i)} of Proposition \ref{prop:essence_of_transference_split_up_weighted}, Lemma \ref{l:for_omega_times_omega}, and Proposition \ref{prop:dual_weighted_in_terms_of_schmimmerer}, we get
  \begin{multline*}
    \omega_{\pmb\sigma,\pmb\rho}^{(k)}(\Theta)=\gamma
    \iff
    \bpsi_k(\La,\pmb\mu(\gamma))=1-\gamma
    \iff
    \apsi_{d+1-k}(\La^\ast,-\pmb\mu(\gamma))=\gamma-1
    \iff \\
    \iff
    \apsi_{d+1-k}\big(\La^\ast,\pmb\mu^\ast\big(\gamma^{-1}\big)\big)=1-\gamma^{-1}
    \iff
    \hat\omega_{\pmb\rho,\pmb\sigma}^{(d+1-k)}(\tr\Theta)=\gamma^{-1},
  \end{multline*}
  Taking into account Corollary \ref{cor:weighted_in_terms_of_schmimmerer}, we get in a similar way that
  \[
    \omega_{\pmb\sigma,\pmb\rho}^{(k)}(\Theta)<\gamma
    \iff
    \hat\omega_{\pmb\rho,\pmb\sigma}^{(d+1-k)}(\tr\Theta)>\gamma^{-1},
    \quad
    \omega_{\pmb\sigma,\pmb\rho}^{(k)}(\Theta)>\gamma
    \iff
    \hat\omega_{\pmb\rho,\pmb\sigma}^{(d+1-k)}(\tr\Theta)<\gamma^{-1},
  \]
  Thus, \eqref{eq:omega_times_omega} holds for all possible values of the factors, including zero and $+\infty$.
\end{proof}

\subsection{Concerning inhomogeneous approximation}\label{sec:inhomogeneous}

There is an important class of Diophantine problems that concerns the inhomogeneous setting. Given $\pmb\eta=(\eta_1,\ldots,\eta_n)\in\R^n$, consider the system
\begin{equation}\label{eq:inhomogeneous_system_with_weights}
  \begin{cases}
    |\vec x|_{\pmb\sigma}\leq t \\
    |\Theta\vec x-\vec y-\pmb\eta|_{\pmb\rho}\leq t^{-\gamma}
  \end{cases}
\end{equation}
instead of \eqref{eq:system_with_weights}. The \emph{inhomogeneous weighted Diophantine exponent} $\omega_{\pmb\sigma,\pmb\rho}(\Theta,\pmb\eta)$ (resp. the \emph{inhomogeneous uniform weighted Diophantine exponent} $\hat\omega_{\pmb\sigma,\pmb\rho}(\Theta,\pmb\eta)$) is defined as the supremum of real $\gamma$ such that the system \eqref{eq:inhomogeneous_system_with_weights} admits nonzero solutions in $(\vec x,\vec y)\in\Z^{m+n}$ for some arbitrarily large $t$ (resp. for every $t$ large enough).

In \cite{ghosh_marnat_Pisa_2019} S.\,Chow, A.\,Ghosh et al. proved the following inequalities, the ``non-weighted'' version of which belongs to M.\,Laurent and Y.\,Bugeaud \cite{bugeaud_laurent_2005}:
\begin{equation}\label{eq:weighted_bugeaud_laurent}
  \omega_{\pmb\sigma,\pmb\rho}(\Theta,\pmb\eta)\geq
  \frac1{\hat\omega_{\pmb\rho,\pmb\sigma}(\tr\Theta)},\qquad
  \hat\omega_{\pmb\sigma,\pmb\rho}(\Theta,\pmb\eta)\geq
  \frac1{\omega_{\pmb\rho,\pmb\sigma}(\tr\Theta)},
\end{equation}
where $\hat\omega_{\pmb\rho,\pmb\sigma}(\tr\Theta)$ stands for $\hat\omega_{\pmb\rho,\pmb\sigma}^{(1)}(\tr\Theta)$.

In view of Theorem \ref{t:omega_times_omega} \eqref{eq:weighted_bugeaud_laurent} is equivalent to
\begin{equation}\label{eq:weighted_bugeaud_laurent_reformulated}
  \omega_{\pmb\sigma,\pmb\rho}(\Theta,\pmb\eta)\geq
  \omega_{\pmb\sigma,\pmb\rho}^{(d)}(\Theta),\qquad
  \hat\omega_{\pmb\sigma,\pmb\rho}(\Theta,\pmb\eta)\geq
  \hat\omega_{\pmb\sigma,\pmb\rho}^{(d)}(\Theta).
\end{equation}
In such a form the inequalities are very easy to prove with the help of the classical argument, that goes back to A.\,Ya.\,Khintchine and V.\,Jarn\'ik.
Indeed, if $\omega_{\pmb\sigma,\pmb\rho}^{(d)}(\Theta)>\gamma$ (resp. if $\hat\omega_{\pmb\sigma,\pmb\rho}^{(d)}(\Theta)>\gamma$), then there is a constant $c>0$ such that the parallelepiped $c\cP(s,\gamma)$ (in the notation of the proof of Proposition \ref{prop:weighted_in_terms_of_schmimmerer}) contains a fundamental parallelepiped of $\La$ for some arbitrarily large $s$ (resp. for every $s$ large enough). Therefore, given arbitrary $\pmb\zeta\in\R^d$, the shifted parallelepiped $c\cP(s,\gamma)+\pmb\zeta$ contains a point of $\La$ for some arbitrarily large $s$ (resp. for every $s$ large enough). Hence $\omega_{\pmb\sigma,\pmb\rho}(\Theta,\pmb\eta)\geq\gamma$ (resp. $\hat\omega_{\pmb\sigma,\pmb\rho}(\Theta,\pmb\eta)\geq\gamma$), and \eqref{eq:weighted_bugeaud_laurent_reformulated} follows.

Thus, \eqref{eq:weighted_bugeaud_laurent} is a consequence of Theorem \ref{t:omega_times_omega}.

\section{Lattice exponents and multiparametric geometry on numbers}\label{sec:lattice_exp}

\subsection{Intermediate exponents}\label{sec:lattice_inter_exp}

Let $\La$ be an arbitrary full rank lattice in $\R^d$ of covolume $1$. Let also $\Pi(\cdot)$ be as defined in Section \ref{sec:lattice_exp_intro}.

Same as in the case of Diophantine approximation with weights, let us supplement $\omega(\La)$ with two families of intermediate exponents. For every $\vec v=(v_1,\ldots,v_d)\in\R^d$ let us denote
\[
  \cP(\vec v)=\Big\{ \vec z=(z_1,\ldots,z_d)\in\R^d \ \Big|\ |z_i|\leq|v_i|,\ i=1,\ldots,d \Big\}.
\]

\begin{definition}\label{def:intermediate_ordinary_lattice_exp}
  Let $k$ be an integer, $1\leq k\leq d$.
  We define the $k$-th \emph{Diophantine exponent} $\omega_k(\La)$
%  \emph{of the first type}
  as the supremum of real $\gamma$ such that $\cP(\vec v)$ contains $k$ linearly independent points of $\La$ for some $\vec v\in\R^d$ satisfying $\Pi(\vec v)=|\vec v|^{-\gamma}$ with $|\vec v|$ however large.
\end{definition}

Clearly, $\omega(\La)=\omega_1(\La)$.

\begin{definition}\label{def:intermediate_uniform_lattice_exp}
  Let $k$ be an integer, $1\leq k\leq d$.
  We define the $k$-th \emph{uniform Diophantine exponent} $\hat\omega_k(\La)$
%  \emph{of the first type}
  as the supremum of real $\gamma$ such that $\cP(\vec v)$ contains $k$ linearly independent points of $\La$ for every $\vec v\in\R^d$ satisfying $\Pi(\vec v)=|\vec v|^{-\gamma}$ with $|\vec v|$ large enough.
\end{definition}

Let us interpret the exponents just defined in terms of multiparametric geometry of numbers.

Every norm in $\R^d$ induces a norm in $\cT$. Particularly, the supremum norm $|\cdot|$. As for the functionals induced by $|\cdot|_+$ and $|\cdot|_-$ (defined at the beginning of Section \ref{sec:properties_of_L_and_S}), they are not norms for $d\geq3$, the corresponding ``unit balls'' are simplices and are not symmetric w.r.t. the origin. However, $|\cdot|_+$ cannot be neglected, as it is the image of the supremum norm under the logarithmic mapping: if $\vec z=(z_1,\ldots,z_d)$, $z_i>0$, $i=1,\ldots,d$, and $\vec z_{\log}=(\log z_1,\ldots,\log z_d)$, then
\[\log|\vec z|=|\vec z_{\log}|_+\,.\]
The functional $|\cdot|_+$ will be most important for us in this Section. However, some of the statements we prove are valid for an arbitrary functional one can choose to measure $\pmb\tau$, it only needs to generate an exhaustion of $\cT$.

Let $f$ be an arbitrary non-negative function on $\cT$ such that the sets
\[\Big\{ \pmb\tau\in\cT\, \Big|\,f(\pmb\tau)\leq\lambda \Big\}\]
form a monotone exhaustion of $\cT$.

\begin{definition}\label{def:psi_lattice}
  Given $\La$, $f$, and $k\in\{1,\ldots,d\}$, the quantities
  \[\bpsi_k(\La,f)=\liminf_{|\pmb\tau|\to\infty}\frac{L_k(\La,\pmb\tau)}{f(\pmb\tau)}\,,
    \qquad
    \apsi_k(\La,f)=\limsup_{|\pmb\tau|\to\infty}\frac{L_k(\La,\pmb\tau)}{f(\pmb\tau)}\]
  are called the \emph{Schmidt--Summerer lower and upper exponents of the first type}, and the quantities
  \[\bPsi_k(\La,f)=\liminf_{|\pmb\tau|\to\infty}\frac{S_k(\La,\pmb\tau)}{f(\pmb\tau)},
    \qquad
    \aPsi_k(\La,f)=\limsup_{|\pmb\tau|\to\infty}\frac{S_k(\La,\pmb\tau)}{f(\pmb\tau)}\]
  are called the \emph{Schmidt--Summerer lower and upper exponents of the second type}.
\end{definition}

\begin{proposition} \label{prop:omega_vs_psi}
  Let $f(\pmb\tau)=|\pmb\tau|_+$. Then, for each $k=1,\ldots,d$, we have
  \begin{align}
    \label{eq:omega_geq_-1}
    \omega_k(\La)\geq
    \hat\omega_k(\La)\geq
    -1+\frac1d\,, \\
    \label{eq:psi_geq_-1}
    -1\leq
    \bpsi_k(\La,f)\leq
    \apsi_k(\La,f)\leq
    d-1,
  \end{align}
  and
  \begin{equation}\label{eq:omega_vs_psi}
  \begin{array}{l}
    \big(1+\omega_k(\La)\big)\big(1+\bpsi_k(\La,f)\big)=1, \\
    \big(1+\hat\omega_k(\La)\big)\big(1+\apsi_k(\La,f)\big)=1,
    \vphantom{1^{\big|}}
  \end{array}
  \end{equation}
  assuming that $\omega_k(\La)=+\infty$ whenever $\bpsi_k(\La,f)=-1$, and $\hat\omega_k(\La)=+\infty$ whenever $\apsi_k(\La,f)=-1$.
\end{proposition}

\begin{proof}
  For each $\gamma\in\R$ let us set
  \[
    \cH_\gamma=\Big\{ \vec v\in\R_{>0}^d \,\Big|\, \Pi(\vec v)=|\vec v|^{-\gamma} \Big\}
  \]
  By Definitions \ref{def:intermediate_ordinary_lattice_exp}, \ref{def:intermediate_uniform_lattice_exp} the exponent $\omega_k(\La)$ (resp. $\hat\omega_k(\La)$) equals the supremum of $\gamma\in\R$ such that $\cP(\vec v)$ contains $k$ linearly independent points of $\La$ for some $\vec v\in\cH_\gamma$ with $|\vec v|$ however large (resp. for every $\vec v\in\cH_\gamma$ with $|\vec v|$ large enough). By Definition \ref{def:psi_lattice} the exponent $\bpsi_k(\La,f)$ (resp. $\apsi_k(\La,f)$) equals the infimum of $\chi\in\R$ such that $e^{f(\pmb\tau)\chi}\cB_{\pmb\tau}$ contains $k$ linearly independent points of $\La$ for some $\pmb\tau\in\cT$ with $f(\pmb\tau)$ however large (resp. for every $\pmb\tau\in\cT$ with $f(\pmb\tau)$ large enough).

  For every $\gamma>-1$ let us define $\chi(\gamma)$ by the relation
  \[
    (1+\chi(\gamma))(1+\gamma)=1.
  \]
  Let us also consider the bijection between $\cH_\gamma$ and $\cT$ determined by
  \[
    \vec v=(v_1,\ldots,v_d)\mapsto\pmb\tau(\vec v)=\Big(\log\big(v_1\big/\Pi(\vec v)\big),\ldots,\log\big(v_d\big/\Pi(\vec v)\big)\Big).
  \]
  Then for every $\vec v\in\cH_\gamma$ with $\gamma>-1$
  \begin{equation}\label{eq:f_chi_Pi}
    f(\pmb\tau(\vec v))\chi(\gamma)=
    \log(\Pi(\vec v)).
  \end{equation}
  Indeed, if $\gamma=0$, then both sides of \eqref{eq:f_chi_Pi} are equal to zero, whereas if $\gamma\neq0$, we have
  \begin{multline*}
    f(\pmb\tau(\vec v))=
    |\pmb\tau(\vec v)|_+=
    \log|\vec v|-\log(\Pi(\vec v))= \\
    =-(1+\gamma^{-1})\log(\Pi(\vec v))=
    \chi(\gamma)^{-1}\log(\Pi(\vec v)).
  \end{multline*}
  It follows from \eqref{eq:f_chi_Pi} that for every $\gamma>-1$ and every $\vec v\in\cH_\gamma$
  \[
    \cP(\vec v)=
    \Pi(\vec v)\cB_{\pmb\tau(\vec v)}=
    e^{f(\pmb\tau(\vec v))\chi(\gamma)}\cB_{\pmb\tau(\vec v)}.
  \]
  Hence, taking into account the reformulation of Definitions \ref{def:intermediate_ordinary_lattice_exp}, \ref{def:intermediate_uniform_lattice_exp}, \ref{def:psi_lattice} given above, we obtain that for every $\gamma>-1$
  \begin{equation}\label{eq:omega_iff_psi}
  \begin{array}{l}
    \omega_k(\La)=\gamma
    \iff
    \bpsi_k(\La,f)=\chi(\gamma), \\
    \hat\omega_k(\La)=\gamma
    \iff
    \apsi_k(\La,f)=\chi(\gamma).
    \vphantom{1^{\big|}}
  \end{array}
  \end{equation}
  Furthermore, it follows from statements \textup{(ii)}, \textup{(iii)} of Proposition \ref{prop:properties_of_L_k}, and relation \eqref{eq:tau_minus_vs_tau_plus} that
  \[
    \begin{array}{l}
      L_1(\pmb\tau)\big/|\pmb\tau|_+
      \geq
      -1+o(1), \\
      L_d(\pmb\tau)\big/|\pmb\tau|_+
      \leq
      |\pmb\tau|_-\big/|\pmb\tau|_++o(1)
      \leq
      d-1+o(1),
      \vphantom{1^{\big|}}
    \end{array}
  \]
  as $|\pmb\tau|\to\infty$. Hence \eqref{eq:psi_geq_-1} immediately follows. Since the correspondence $\gamma\to\chi(\gamma)$ is a bijection between $[-1+1/d,+\infty]$ and $[-1,d-1]$, \eqref{eq:omega_iff_psi} implies both \eqref{eq:omega_vs_psi} and \eqref{eq:omega_geq_-1}.
\end{proof}

\subsection{Intermediate exponents of the second type}\label{sec:lattice_inter_exp_second_type}

In accordance with the notation introduced in Section \ref{sec:weighted_inter_exp_second_type}, let us set for each $\vec v=(v_1,\ldots,v_d)\in\R^d$
\[
  \cP^{[k]}(\vec v)=\Big\{ \vec Z=(Z_{i_1,\ldots,i_k})\in{\textstyle\bigwedge^k}(\R^d) \ \Big|\ |Z_{i_1,\ldots,i_k}|\leq\prod_{j=1}^{k} |v_{i_j}| \Big\}.
\]
Let us also consider the following functionals generalizing both the supremum norm and $\Pi(\cdot)$. Given $\vec v=(v_1,\ldots,v_d)\in\R^d$, set
\[
  |\vec v|^{[k]}=\max_{1\leq i_1<\ldots<i_k\leq d}\prod_{1\leq j\leq k} |v_{i_j}|^{1/k}.
\]
Clearly, $|\vec v|^{[1]}=|\vec v|$ and $|\vec v|^{[d]}=\Pi(\vec v)$.

\begin{definition}\label{def:intermediate_ordinary_lattice_exp_second_type}
  Let $k$ be an integer, $1\leq k\leq d$.
  We define the $k$-th \emph{Diophantine exponent} $\Omega_k(\La)$ \emph{of the second type} as the supremum of real $\gamma$ such that $\cP^{[k]}(\vec v)$ contains a nonzero element of $\bigwedge^k(\La)$ for some $\vec v\in\R_{>0}^d$ satisfying $\Pi(\vec v)=\big(|\vec v|^{[k]}\big)^{-\gamma}$ with $|\vec v|$ however large.
\end{definition}

\begin{definition}\label{def:intermediate_uniform_lattice_exp_second_type}
  Let $k$ be an integer, $1\leq k\leq d$.
  We define the $k$-th \emph{uniform Diophantine exponent} $\hat\Omega_k(\La)$ \emph{of the second type} as the supremum of real $\gamma$ such that $\cP^{[k]}(\vec v)$ contains a nonzero element of $\bigwedge^k(\La)$ for every $\vec v\in\R_{>0}^d$ satisfying $\Pi(\vec v)=\big(|\vec v|^{[k]}\big)^{-\gamma}$ with $|\vec v|$ large enough.
\end{definition}

Clearly, $\Omega_1(\La)=\omega_1(\La)=\omega(\La)$.

Furthermore, since $\det\big(\bigwedge^k(\La)\big)=1$ and $\vol\big(\cP^{[k]}(\vec v)\big)=2^{\binom dk}$, provided $\Pi(\vec v)=1$, it follows by Minkowski's convex body theorem that
\begin{equation}\label{eq:Omega_geq_0}
  \Omega_k(\La)\geq\hat\Omega_k(\La)\geq0.
\end{equation}
Moreover, for $k=d$ the exponents degenerate: we have
\begin{equation*}%\label{eq:Omega_geq_0}
  \Omega_d(\La)=\hat\Omega_d(\La)=0.
\end{equation*}

In order to prove an analogue of Proposition \ref{prop:omega_vs_psi} for the exponents of the second type, let us consider the image of $|\cdot|^{[k]}$ under the logarithmic mapping. Let us set
\[
  |\pmb\tau|_+^{[k]}=
  \max_{1\leq i_1<\ldots<i_k\leq d}
  \frac{\pmb\tau_{i_1}+\ldots+\pmb\tau_{i_k}}{k}.
\]
Clearly, $|\pmb\tau|_+^{[1]}=|\pmb\tau|_+$ and $|\pmb\tau|_+^{[d]}=0$ for every $\pmb\tau\in\cT$. It is also clear that each $|\cdot|_+^{[k]}$ with $k\in\{1,\ldots,d-1\}$ generates an exhaustion of $\cT$, whereas $|\cdot|_+^{[d]}$ does not.

\begin{proposition} \label{prop:Omega_vs_Psi}
  Let $k\in\{1,\ldots,d-1\}$ and let $f(\pmb\tau)=|\pmb\tau|_+^{[k]}$. Then, along with \eqref{eq:Omega_geq_0}, we have
  \begin{equation*}%\label{eq:Psi_geq_-1}
    -k\leq
    \bPsi_k(\La,f)\leq
    \aPsi_k(\La,f)\leq0
  \end{equation*}
  and
  \begin{equation*}%\label{eq:Omega_vs_Psi}
  \begin{array}{l}
    \big(1+\Omega_k(\La)\big)
    \bigg(1+\dfrac{\bPsi_k(\La,f)}{k}\bigg)=1, \\
    \big(1+\hat\Omega_k(\La)\big)
    \bigg(1+\dfrac{\aPsi_k(\La,f)}{k}\bigg)=1,
    \vphantom{1^{\big|}}
  \end{array}
  \end{equation*}
  assuming that $\Omega_k(\La)=+\infty$ whenever $\bPsi_k(\La,f)=-k$, and $\hat\Omega_k(\La)=+\infty$ whenever $\aPsi_k(\La,f)=-k$.
\end{proposition}

\begin{proof}
  It follows from Mahler’s theory of compound bodies that
  \[
    \lambda_1\Big(\cB^{[k]}_{\pmb\tau},{\textstyle\bigwedge^k}(\La)\Big)
    \asymp
    \prod_{i=1}^{k}\lambda_i\big(\cB_{\pmb\tau},\La\big)
  \]
  with the implied constant depending only on d.

  Thus, the exponent $\bPsi_k(\La,f)$ (resp. $\aPsi_k(\La,f)$) equals the infimum of $\chi\in\R$ such that $e^{f(\pmb\tau)\chi}\cB^{[k]}_{\pmb\tau}$ contains a nonzero element of $\bigwedge^k(\La)$ for some arbitrarily large $|\pmb\tau|$ (resp. for every $|\pmb\tau|$ large enough). Since $\det\big(\bigwedge^k(\La)\big)=1$, it follows by Minkowski's convex body theorem that $\bPsi_k(\La,f)\leq\aPsi_k(\La,f)\leq0$.

  Following the ideas of the proof of Proposition \ref{prop:omega_vs_psi}, let us set for each $\gamma\geq0$
  \[
    \cH_\gamma=\Big\{ \vec v\in\R_{>0}^d \,\Big|\, \Pi(\vec v)=\big(|\vec v|^{[k]}\big)^{-\gamma} \Big\}
  \]
  Since $1\leq k\leq d-1$, it is easily verified that for each $\vec v\in\R^d$ with nonzero $\Pi(\vec v)$ there is a unique positive $\lambda$ such that $\Pi(\lambda\vec v)=\big(|\lambda\vec v|^{[k]}\big)^{-\gamma}$.
  By Definitions \ref{def:intermediate_ordinary_lattice_exp_second_type}, \ref{def:intermediate_uniform_lattice_exp_second_type} the exponent $\Omega_k(\La)$ (resp. $\hat\Omega_k(\La)$) equals the supremum of $\gamma\in\R$ such that $\cP^{[k]}(\vec v)$ contains a nonzero element of $\bigwedge^k(\La)$ for some $\vec v\in\cH_\gamma$ with $|\vec v|$ however large (resp. for every $\vec v\in\cH_\gamma$ with $|\vec v|$ large enough).

  For every $\gamma\geq0$ let us define $\chi(\gamma)$ by the relation
  \[
    (1+\chi(\gamma))(1+\gamma)=1.
  \]
  Let us also consider the same bijection between $\cH_\gamma$ and $\cT$, as in the proof of Proposition \ref{prop:omega_vs_psi}, i.e.
  \[
    \vec v=(v_1,\ldots,v_d)\mapsto\pmb\tau(\vec v)=\Big(\log\big(v_1\big/\Pi(\vec v)\big),\ldots,\log\big(v_d\big/\Pi(\vec v)\big)\Big).
  \]
  Then, for each $\gamma>0$,
  \begin{multline*}
    f(\pmb\tau(\vec v))=
    |\pmb\tau(\vec v)|_+^{[k]}=
    \log|\vec v^{[k]}|-\log(\Pi(\vec v))= \\
    =-(1+\gamma^{-1})\log(\Pi(\vec v))=
    \chi(\gamma)^{-1}\log(\Pi(\vec v)).
  \end{multline*}
  Hence, same as \eqref{eq:f_chi_Pi}, we have for each $\gamma\geq0$
  \[
    f(\pmb\tau(\vec v))\chi(\gamma)=
    \log(\Pi(\vec v)).
  \]
  Thus,
  \[
    \cP^{[k]}(\vec v)=
    \Pi(\vec v)^k\cB_{\pmb\tau(\vec v)}^{[k]}=
    e^{kf(\pmb\tau(\vec v))\chi(\gamma)}\cB_{\pmb\tau(\vec v)}^{[k]}.
  \]
  Therefore, for every $\gamma\geq0$
  \begin{equation*}%\label{eq:Omega_iff_Psi}
  \begin{array}{l}
    \Omega_k(\La)=\gamma
    \iff
    \bPsi_k(\La,f)=k\chi(\gamma), \\
    \hat\Omega_k(\La)=\gamma
    \iff
    \aPsi_k(\La,f)=k\chi(\gamma).
    \vphantom{1^{\big|}}
  \end{array}
  \end{equation*}
  It remains to make use of \eqref{eq:Omega_geq_0}.
\end{proof}

\subsection{Application of general theory and splitting Theorem \ref{t:lattice_transference}}\label{sec:application_lattice_exp}

Let us set
\[f^\ast(\pmb\tau)=f(-\pmb\tau).\]
Clearly, $|\pmb\tau|\to\infty$ if and only if $f(\pmb\tau)\to\infty$. Hence, dividing all the relations provided by Proposition \ref{prop:essence_of_transference_split_up} and Corollary \ref{cor:essence_of_transference} by $f(\pmb\tau)$ and sending $\pmb\tau$ to infinity, we get the following statements on the Schmidt--Summerer exponents, analogous to Proposition \ref{prop:essence_of_transference_split_up_weighted} and Corollary \ref{cor:essence_of_transference_weighted}.

\begin{proposition} \label{prop:essence_of_transference_split_up_lattice_exp}
  Given $\La$ and $f$, we have

  \textup{(i)}
  $\bpsi_k(\La,f)=-\apsi_{d+1-k}(\La^\ast,f^\ast)$, \vphantom{$\frac{\big|}{|}$}

  \phantom{\textup{(i)}}
  $\apsi_k(\La,f)=-\bpsi_{d+1-k}(\La^\ast,f^\ast);\ \ k=1,\ldots,d$; \vphantom{$\Big|$}

  \textup{(ii)}
  $\bPsi_k(\La,f)=\bPsi_{d-k}(\La^\ast,f^\ast)$,
  \vphantom{$\frac{\big|}{|}$}

  \phantom{\textup{(ii)}}
  $\aPsi_k(\La,f)=\aPsi_{d-k}(\La^\ast,f^\ast),\ \ k=1,\ldots,d-1$; \vphantom{$\Big|$}

  \textup{(iii)}
  $
    \displaystyle
    \bPsi_1(\La,f)
    \leq\ldots\leq
    \frac{\bPsi_k(\La,f)}{k}
    \leq\ldots\leq
    \frac{\bPsi_{d-1}(\La,f)}{d-1}
    \leq
    \frac{\bPsi_d(\La,f)}{d}
    =0
  $,

  \phantom{\textup{(iii)}}
  $
    \displaystyle
    \aPsi_1(\La,f)
    \leq\ldots\leq
    \frac{\aPsi_k(\La,f)}{k}
    \leq\ldots\leq
    \frac{\aPsi_{d-1}(\La,f)}{d-1}
    \leq
    \frac{\aPsi_d(\La,f)}{d}
    =0
  $;
  \vphantom{$\bigg|$}

  \textup{(iv)}
  $
    \displaystyle
    \frac{\bPsi_1(\La,f)}{d-1}
    \geq\ldots\geq
    \frac{\bPsi_k(\La,f)}{d-k}
    \geq\ldots\geq
    \bPsi_{d-1}(\La,f)
  $,
  \vphantom{$\bigg|$}

  \phantom{\textup{(iv)}}
  $
    \displaystyle
    \frac{\aPsi_1(\La,f)}{d-1}
    \geq\ldots\geq
    \frac{\aPsi_k(\La,f)}{d-k}
    \geq\ldots\geq
    \aPsi_{d-1}(\La,f)
  $.
\end{proposition}

\begin{corollary} \label{cor:essence_of_transference_lattice_exp}
  $
    \bPsi_1(\La,f)
    \leq
    \dfrac{\bPsi_1(\La^\ast,f^\ast)}{d-1}
  $.
\end{corollary}

As in the case of Diophantine approximation with weights, we claim that Corollary \ref{cor:essence_of_transference_lattice_exp} implies Theorem \ref{t:lattice_transference}. To show this, we need the following simple observation.

\begin{lemma}\label{l:f_ast_vs_f}
  Let $f(\pmb\tau)=|\pmb\tau|_+$. Then, for each $k\in\{1,\ldots,d\}$, the signs of $\bpsi_k(\La^\ast,f)$ and $\bpsi_k(\La^\ast,f^\ast)$ coincide, as well as do the signs of $\apsi_k(\La^\ast,f)$ and $\apsi_k(\La^\ast,f^\ast)$. Furthermore, we have
  \begin{equation}\label{eq:f_ast_vs_f}
  \begin{array}{l}
    \dfrac{|\bpsi_k(\La^\ast,f)|}{d-1}\leq
    |\bpsi_k(\La^\ast,f^\ast)|\leq
    (d-1)|\bpsi_k(\La^\ast,f)|, \\
    \dfrac{|\apsi_k(\La^\ast,f)|}{d-1}\leq
    |\apsi_k(\La^\ast,f^\ast)|\leq
    (d-1)|\apsi_k(\La^\ast,f)|.
    \vphantom{1^{\Big|}}
  \end{array}
  \end{equation}
\end{lemma}

\begin{proof}
  The sign of $\bpsi_k(\La^\ast,f)$ (resp. $\apsi_k(\La^\ast,f)$) does not depend on the choice of $f$, it depends only on whether $L_k(\pmb\tau)\geq0$ for some $\pmb\tau$ with $|\pmb\tau|$ however large (resp. for every $\pmb\tau$ with $|\pmb\tau|$ large enough).

  Furthermore, it follows from \eqref{eq:tau_minus_vs_tau_plus} that
  \begin{equation}\label{eq:f_ast_vs_f_reason_for_inequality}
    \frac{f(\pmb\tau)}{d-1}\leq f^\ast(\pmb\tau)\leq(d-1)f(\pmb\tau).
  \end{equation}
  This immediately implies \eqref{eq:f_ast_vs_f}.
\end{proof}

\begin{theorem}\label{t:lattice_transference_schmimmerered}
  Let $f(\pmb\tau)=|\pmb\tau|_+$. Then
  \[
    \bpsi_1(\La,f)\leq
    \dfrac{\bpsi_1(\La^\ast,f)}{(d-1)^2}.
  \]
\end{theorem}

\begin{proof}
  The exponents $\bpsi_1$ and $\bPsi_1$ are the same. By statement \textup{(iii)} of Proposition \ref{prop:essence_of_transference_split_up_lattice_exp} they are nonpositive. It remains to apply Lemma \ref{l:f_ast_vs_f} with $k=1$ and Corollary \ref{cor:essence_of_transference_lattice_exp}.
\end{proof}

Similar to the case of Diophantine approximation with weights, Theorem \ref{t:lattice_transference_schmimmerered} is a reformulation of Theorem \ref{t:lattice_transference}, due to Proposition \ref{prop:omega_vs_psi} and the observation that for nonzero $x,y\in\R$ the relation $(1+x)(1+y)=1$ is equivalent to $x^{-1}+y^{-1}+1=0$. The key ingredient in the proof of Theorem \ref{t:lattice_transference_schmimmerered} is Corollary \ref{cor:essence_of_transference_lattice_exp}. In its turn, Corollary \ref{cor:essence_of_transference_lattice_exp} gets split by statements \textup{(ii)} and \textup{(iii)} of Proposition \ref{prop:essence_of_transference_split_up_lattice_exp} into a sequence of inequalities between the Schmidt--Summerer exponents $\bPsi_k(\La,f)$. We should note, however, that such a splitting involves only $f=|\cdot|_+$ and thus cannot be immediately interpreted as a result concerning lattice exponents of the second type. In order to perform such a result, let us make some observations concerning the functionals $|\cdot|^{[k]}_+$.

\begin{lemma} \label{l:k-th_functional_vs_(d-k)-th}
  For each $\pmb\tau\in\cT$ and each $k=1,\ldots,d-1$ we have
  \[
    |\pmb\tau|_+^{[k]}=\frac{d-k}{k}\cdot|-\pmb\tau|_+^{[d-k]}.
  \]
  In other words, if $f=|\cdot|_+^{[k]}$, then $f^\ast=\dfrac{d-k}{k}|\cdot|_+^{[d-k]}$.
\end{lemma}

\begin{proof}
  Suppose $\pmb\tau=(\tau_1,\ldots,\tau_d)$, $\tau_1\geq\ldots\geq\tau_d$. Then, since $\tau_1+\ldots+\tau_d=0$,
  \[
    |\pmb\tau|_+^{[k]}=
    \frac1k\sum_{j=1}^k\tau_j=
    \frac{d-k}{k}\cdot\frac{1}{d-k}\sum_{j=k+1}^d(-\tau_j)=
    \frac{d-k}{k}\cdot|-\pmb\tau|_+^{[d-k]}.
  \]
\end{proof}

\begin{corollary}\label{cor:k-th_functional_vs_(d-k)-th}
  Let $f_k(\pmb\tau)=|\pmb\tau|_+^{[k]}$, $k=1,\ldots,d-1$. Then, for every such $k$, we have
  \[
    \frac{\bPsi_k(\La,f_k)}{k}=
    \frac{\bPsi_{d-k}(\La^\ast,f_{d-k})}{d-k}.
  \]
\end{corollary}

\begin{proof}
  It suffices to apply statement \textup{(ii)} of Proposition \ref{prop:essence_of_transference_split_up_lattice_exp} and Lemma \ref{l:k-th_functional_vs_(d-k)-th}.
\end{proof}

\begin{lemma} \label{l:k-th_functional_vs_(k+1)-th}
  For each $\pmb\tau\in\cT$ and each $k=1,\ldots,d-1$ we have
  \[
    \frac{k(d-k-1)}{(d-k)(k+1)}|\pmb\tau|_+^{[k]}
    \leq
    |\pmb\tau|_+^{[k+1]}
    \leq
    |\pmb\tau|_+^{[k]}.
  \]
\end{lemma}

\begin{proof}
  Suppose $\pmb\tau=(\tau_1,\ldots,\tau_d)$, $\tau_1\geq\ldots\geq\tau_d$. Then, since $\tau_1+\ldots+\tau_d=0$,
  \[
    \frac1k\sum_{j=1}^k\tau_j
    \geq
    \tau_{k+1}
    \geq
    \frac1{d-k}\sum_{j=k+1}^d\tau_j=
    \frac{-1}{d-k}\sum_{j=1}^k\tau_j.
  \]
  Hence
  \[
    \frac{d-k-1}{d-k}\sum_{j=1}^k\tau_j\leq
    \sum_{j=1}^{k+1}\tau_j\leq
    \frac{k+1}{k}\sum_{j=1}^k\tau_j,
  \]
  and the desired inequality follows.
\end{proof}

\begin{corollary}\label{cor:k-th_functional_vs_(k+1)-th}
  Let $f_k(\pmb\tau)=|\pmb\tau|_+^{[k]}$, $k=1,\ldots,d-1$. Then, for every $k\in\{1,\ldots,d-2\}$, we have
  \[
    \frac{d-k}{k}\cdot
    \frac{\bPsi_k(\La,f_k)}{k}
    \leq
    \frac{d-k-1}{k+1}\cdot
    \frac{\bPsi_{k+1}(\La,f_{k+1})}{k+1}
  \]
\end{corollary}

\begin{proof}
  It suffices to apply statement \textup{(iii)} of Proposition \ref{prop:essence_of_transference_split_up_lattice_exp}, Lemma \ref{l:k-th_functional_vs_(k+1)-th}, and the fact that the Schmidt--Summerer exponents of the second type, as well as the functions $S_k$, are nonpositive (see statement \textup{(i)} of Proposition \ref{prop:properties_of_S_k}).
\end{proof}

Applying Corollaries \ref{cor:k-th_functional_vs_(d-k)-th}, \ref{cor:k-th_functional_vs_(k+1)-th}, Proposition \ref{prop:Omega_vs_Psi}, and the fact that for nonzero $x,y\in\R$ the relation $(1+x)(1+y)=1$ is equivalent to $x^{-1}+y^{-1}+1=0$, we get the following splitting of Theorem \ref{t:lattice_transference}.

\begin{theorem}
  If any of the exponents $\Omega_1(\La),\ldots,\Omega_{d-1}(\La),\Omega_1(\La^\ast),\ldots,\Omega_{d-1}(\La^\ast)$ is zero, then so are all the others. If they are nonzero, then
  \[
    \dfrac{1+\Omega_1(\La)^{-1}}{d-1}
    \leq\ldots\leq
    \dfrac{k}{d-k}
    \big(1+\Omega_k(\La)^{-1}\big)
    \leq\ldots\leq
    (d-1)
    \big(1+\Omega_{d-1}(\La)^{-1}\big)
  \]
  and
  $
    \Omega_{d-1}(\La)=
    \Omega_1(\La^\ast).
  $
\end{theorem}

\subsection{Transference inequalities for intermediate Diophantine exponents of lattices}\label{sec:pseudo_unit_product}

We conclude with another transference result, one concerning the exponents $\bpsi_k(\La,f)$, $\apsi_k(\La,f)$, and thus, the exponents $\omega_k(\La)$, $\hat\omega_k(\La)$. Combining statement \textup{(i)} of Proposition \ref{prop:essence_of_transference_split_up_lattice_exp} with Lemma \ref{l:f_ast_vs_f} immediately produces the following statement.

\begin{theorem}\label{t:small_psi_transference}
  Let $f(\pmb\tau)=|\pmb\tau|_+$. Then, for each $k\in\{1,\ldots,d\}$, the exponents $\bpsi_k(\La,f)$, $\apsi_{d+1-k}(\La^\ast,f)$ have different signs, or are simultaneously zero. Besides that,
  \begin{equation*}%\label{eq:}
  \begin{array}{l}
    |\bpsi_k(\La,f)|
    \leq
    (d-1)|\apsi_{d+1-k}(\La^\ast,f)|, \\
    |\apsi_k(\La,f)|
    \leq
    (d-1)|\bpsi_{d+1-k}(\La^\ast,f)|.
    \vphantom{1^{\big|}}
  \end{array}
  \end{equation*}
\end{theorem}

Applying again Proposition \ref{prop:omega_vs_psi} and the fact that for nonzero $x,y\in\R$ the relation $(1+x)(1+y)=1$ is equivalent to $x^{-1}+y^{-1}+1=0$, we get the following reformulation of Theorem \ref{t:small_psi_transference}.

\begin{theorem}\label{t:omega_times_omega_lattice_exp}
  Let $f(\pmb\tau)=|\pmb\tau|_+$, $k\in\{1,\ldots,d\}$.

  The exponents $\omega_k(\La)$, $\hat\omega_{d+1-k}(\La^\ast)$ have different signs, or are simultaneously zero.

  If they are nonzero, let $A$ be the positive one of them, and let $B$ be the negative one. Then
  \begin{equation}\label{eq:omega_times_omega_lattice_exp}
    \frac{A^{-1}+1}{(d-1)}
    \leq
    -(B^{-1}+1)
    \leq
    (d-1)(A^{-1}+1).
  \end{equation}
\end{theorem}

\begin{remark}
  Theorem \ref{t:omega_times_omega_lattice_exp} is an analogue of Theorem \ref{t:omega_times_omega}. However, one can easily notice that \eqref{eq:omega_times_omega_lattice_exp} is an inequality, whereas \eqref{eq:omega_times_omega} is an equality. The reason for this difference is \eqref{eq:f_ast_vs_f_reason_for_inequality}, implied by the choice of $f$. The latter in its turn is implied by the choice of the supremum norm to define $\omega(\La)$. If $f$ were symmetric, then $f=f^\ast$ would hold instead of \eqref{eq:f_ast_vs_f_reason_for_inequality}, and the factor $(d-1)$ would vanish on both sides of \eqref{eq:omega_times_omega_lattice_exp}, turning \eqref{eq:omega_times_omega_lattice_exp} into
  \[
    A^{-1}+B^{-1}=-2.
  \]
  We must admit though, that no such choice of $f$ seems as natural from the point of view of the definition of lattice exponents, as $f=|\cdot|_+$.
\end{remark}

\paragraph{Acknowledgements.}

The author is a winner of the ``Junior Leader'' contest conducted by Theoretical Physics and Mathematics Advancement Foundation “BASIS” and would like to thank its sponsors and jury.

\end{document}